\documentclass{article}
\usepackage{amsthm,amsfonts,amsmath}
\usepackage{amscd,amssymb,latexsym,epsfig}
\title{A product estimate, the parabolic Weyl lemma
       and applications}
\author{Joa Weber\\
        Humboldt University Berlin\\
       }

\date{9 October 2009
     }
\newtheorem{theorem}{Theorem}[section]
\newtheorem{corollary}[theorem]{Corollary}

\newtheorem{lemma}[theorem]{Lemma}
\newtheorem{proposition}[theorem]{Proposition}

\theoremstyle{definition}

\newtheorem{remark}[theorem]{Remark}

\theoremstyle{remark}
\newtheorem*{notation}{Notation}
\newcommand{\HH}{{\mathbb{H}}}

\newcommand{\N}{{\mathbb{N}}}

\newcommand{\R}{{\mathbb{R}}}

\newcommand{\Z}{{\mathbb{Z}}}
\newcommand{\Cc}{{\mathcal{C}}}   

\newcommand{\Vv}{{\mathcal{V}}}
\newcommand{\Ww}{{\mathcal{W}}}

\newcommand{\INT}{{\rm int}}       
\newcommand{\supp}{{\rm supp}}     
 
\newcommand{\grad}{{\rm grad }}    
\newcommand{\W}{{\rm W}}

\newcommand{\norm}{{\rm norm}}

\newcommand{\eps}{{\varepsilon}}

\def\NABLA#1{{\mathop{\nabla\kern-.5ex\lower1ex\hbox{$#1$}}}}
\def\Nabla#1{\nabla\kern-.5ex{}_{#1}}
\def\Tabla#1{\Tilde\nabla\kern-.5ex{}_{#1}}
\def\abs#1{\mathopen|#1\mathclose|}   
\def\Abs#1{\left|#1\right|}            
\def\norm#1{\mathopen\|#1\mathclose\|}
\def\Norm#1{\left\|#1\right\|}

\renewcommand{\Tilde}{\widetilde}

\newcommand{\p}{{\partial}}

\begin{document}
\maketitle


\begin{abstract}
We prove a product estimate
that allows to estimate
the quadratic first order nonlinearity
of the harmonic map flow
in the $L^p$ norm.
Then the parabolic analogue
of Weyl's lemma for the Lapace operator
is established.
Both results are applied to
prove regularity for the heat
flow by parabolic bootstrapping.
\end{abstract}


\section{Introduction and main results}
\label{sec:intro}

There are two main results.
The first result is a product estimate. 
It allows to estimate
the quadratic first order nonlinearity
of the harmonic map flow
in the $L^p$ norm
instead of the $L^{p/2}$ norm --
as one expects at first sight.
A second application, crucial in~\cite{Joa-HEATFLOW},
is to obtain quadratic estimates 
sharp enough to prove
a refined implicit function theorem.

Throughout we identify $S^1=\R/\Z$
and think of $v\in C^\infty(\R\times S^1)$
as a smooth function $v:\R\times\R\to\R$
which satisfies $v(s,t+1)=v(s,t)$.

\begin{theorem}\label{thm:product-estimate}
Fix $2\le p<\infty$.
Then there is a
positive constant
$C_p$ such that
\begin{equation*}
\begin{split}
    &\left(\int_{-T}^0\int_0^1
     \left(\Abs{\p_tv}
     \Abs{\p_tw}\right)^p\, dtds 
     \right)^{1/p}\\
    &\le C_p
     \left(\Norm{v}_p+\Norm{\p_sv}_p
     +\Norm{\p_t\p_tv}_p\right)
     \left(\Norm{\p_tw}_p
     +\Norm{\p_t\p_tw}_p\right)
\end{split}
\end{equation*}
for all compactly supported
smooth maps
$v,w:(-T,0]\times S^1\to\R^k$.
\end{theorem}

The second result
is a parabolic analogue of the Weyl lemma 
in the theory of elliptic partial
differential equations.
This seems to be folklore, known
to experts but hidden --
if not nonexistent -- in the literature.
By $\HH^-$ we denote
the closed lower half plane,
the set of all reals $(s,t)$ such that
$s\le0$ and $t\in\R$.

\begin{lemma}[Parabolic Weyl lemma]
\label{le:parabolic-weyl}
Let $\Omega\subset\HH^-$ be an open subset.
If $u\in L^1_{loc}(\Omega)$ satisfies
\begin{equation}\label{eq:weak-par}
     \int_\Omega u
     \left(-\p_s\phi-\p_t\p_t\phi\right)
     =0
\end{equation}
for every $\phi\in C_0^\infty(\INT\,\Omega)$,
then $u\in C^\infty(\Omega)$
and $\p_su-\p_t\p_tu=0$ on $\Omega$.
\end{lemma}

\begin{notation}
For $T>T^\prime>0$ abbreviate
\begin{equation}\label{eq:Z}
     Z=Z_T=(-T,0]\times S^1,\qquad
     Z^\prime=Z_{T^\prime}=(-T^\prime,0]\times S^1.
\end{equation}
To simplify notation
we denote the anisotropic
Sobolev spaces $W^{k,2k}_p$
by $\Ww^{k,p}$.
More precisely,
fix an integer $k\ge0$, a constant $p\ge1$,
and the domain $\R\times S^1$.
Now set $\Ww^{0,p}=L^p$
and denote by $\Ww^{1,p}$
the set of all $u\in L^p$
which admit weak derivatives
$\p_su$, $\p_tu$, and $\p_t\p_tu$ in $L^p$.
For $k\ge2$ define
$$
     \Ww^{k,p}=\{ u\in \Ww^{1,p}\mid
     \p_su,\p_tu,\p_t\p_tu\in\Ww^{k-1,p}\}
$$
where the derivatives are again meant in
the weak sense.
The associated norm
$$
     \Norm{u}_{\Ww^{k,p}}
     =\left(
     \int_{\R\times S^1}
     \sum_{2\nu+\mu\le 2k}
     \Abs{\p_s^\nu\p_t^\mu u}^p
     \right)^{1/p}
$$
gives $\Ww^{k,p}$ the structure
of a Banach space.
Note the difference to
(standard) Sobolev space $W^{k,p}$
with norm
$\Norm{u}_{\Ww^{k,p}}^p
=\int_{\R\times S^1}\sum_{\nu+\mu\le k}
\Abs{\p_s^\nu\p_t^\mu u}^p$.
\end{notation}

The parabolic Weyl lemma is the key ingredient
to prove part~a) of the next theorem.
The proof of part~b) is based on
theorem~\ref{thm:kerD-CalZyg}
the parabolic analogue
of the Calderon-Zygmund inequality.

\begin{theorem}[Interior regularity]
\label{thm:interior-regularity}
Fix constants $1<q<\infty$ and $T>0$
and an integer $k\ge 0$.
Then the following is true.
\begin{itemize}

\item[\rm a)]
If $u\in L^1_{loc}(Z)$
and $f\in\Ww^{k,q}_{loc}(Z)$ satisfy
\begin{equation}\label{eq:weak-solution}
     \int_{Z}
     u\left(-\p_s\phi-\p_t\p_t\phi\right)
     =\int_{Z} f\phi
\end{equation}
for every $\phi\in C_0^\infty((-T,0)\times S^1)$,
then $u\in\Ww^{k+1,q}_{loc}(Z)$.

\item[\rm b)]
For every
$0<T^\prime<T$
there is a constant
$c=c(k,q,T-T^\prime)$ such that
\begin{equation*}
     \Norm{u}_{\Ww^{k+1,q}(Z^\prime)}
     \le c\left(
     \Norm{\p_su-\p_t\p_tu}_{\Ww^{k,q}(Z)}
     +\Norm{u}_{L^q(Z)}
     \right)
\end{equation*}
for every $u\in C^\infty(\overline{Z})$.
\end{itemize}
\end{theorem}

As an application of
theorem~\ref{thm:product-estimate}
and theorem~\ref{thm:interior-regularity},
hence lemma~\ref{le:parabolic-weyl},
we prove the following
regularity result by
parabolic bootstrapping.

\begin{theorem}[Regularity]
\label{thm:bootstrap}
Fix constants $p>2$, $\mu_0>1$, and $T>0$.
Fix a closed smooth submanifold
$M\hookrightarrow\R^N$ and a
smooth family of vector-valued
symmetric bilinear forms
$\Gamma:M\to\R^{N\times N\times N}$.
Assume that $F:Z\to\R^N$ is a map of class $L^p$
and $u:Z\to\R^N$ is a $\Ww^{1,p}$
map taking values in $M$
with $\norm{u}_{\Ww^{1,p}}\le\mu_0$
such that the perturbed heat equation
\begin{equation}\label{eq:heat-local-F-prime}
     \p_su-\p_t\p_tu
     =\Gamma(u)\left(\p_tu,\p_tu\right)
     +F
\end{equation}
is satisfied almost everywhere.
Then the following is true.
For every integer $k\ge 1$
such that $F\in\Ww^{k,p}(Z)$
and every $T^\prime\in(0,T)$
there is a constant $c_k$ depending on
$k$, $p$, $\mu_0$, $T-T^\prime$,
$\norm{\Gamma}_{C^{2k+2}}$, and
$\norm{F}_{\Ww^{k,p}(Z)}$
such that
$$
     \Norm{u}_{\Ww^{k+1,p}(Z^\prime)}
     \le c_k.
$$
\end{theorem}

The theorem shows that if
$F$ is smooth, then $u$ is smooth
on a slightly smaller domain.
This result is needed in~\cite{Joa-HEATFLOW}
in the case
$F(s,t)=(\grad\,\Vv(u(s,\cdot))(t)$
where $\Vv$ is a smooth function
on the free loop space of $M$.

\vspace{.1cm}
\noindent
{\bf Acknowledgements.} For valuable
comments and discussions the author
would like to thank
K.~Ecker,
T.~Ilmanen,
K.~Mohnke,
J.~Naumann,
D.~Salamon,
and M.~Struwe.
Partial financial support from
SFB~647 is gratefully acknowledged.

%
%
%
%

\section{The product estimate}
\label{sec:product-estimate}

We prove a version
of theorem~\ref{thm:product-estimate}
suitable for global analysis.

\begin{proposition}
\label{prop:product-estimate}
Let $N$ be a Riemannian manifold
with Levi-Civita connection $\nabla$
and Riemannian curvature tensor $R$.
Fix constants $2\le p <\infty$
and $c_0>0$.
Then there is a constant
$C=C(p,c_0,\norm{R}_\infty)>0$
such that
the following holds.
If $u:(a,b]\times S^1\to N$
is a smooth map such that
$
     \Norm{\p_su}_\infty+\Norm{\p_tu}_\infty
     \le c_0
$
then
$$
     \left(\int_a^b\int_0^1
     \left(\Abs{\Nabla{t}\xi}
     \Abs{\Nabla{t}X}\right)^p\, dtds 
     \right)^{1/p}
     \le C\Norm{\xi}_{\Ww^{1,p}}
     \left(\Norm{\Nabla{t} X}_p
     +\Norm{\Nabla{t}\Nabla{t} X}_p\right)
$$
for all smooth
compactly supported
vector fields
$\xi$ and $X$ along $u$.
\end{proposition}

\begin{remark}\label{rmk:product-estimate}
Proposition~\ref{prop:product-estimate}
continues to hold for smooth
maps $u$ that are defined on
the whole cylinder $\R\times S^1$.
In this case
the (compact) supports of
$\xi$ and $X$
are contained in an
interval of the
form $(a,b]$.
\end{remark}

\begin{lemma}[{\cite[lemma~D.4]{SaJoa-LOOP}}]
\label{le:plus-minus}
Let $x\in C^\infty(S^1,M)$
and $p>1$. Then
$$
     \Norm{\Nabla{t}\xi}_p
     \le \kappa_p\left(
     \delta^{-1}\Norm{\xi}_p
     +\delta\Norm{\Nabla{t}\Nabla{t}\xi}_p\right)
$$ 
for $\delta>0$ and smooth
vector fields
$\xi$ along $x$.
Here $\kappa_p$ equals 
$p/(p-1)$ for $p\le 2$
and it equals $p$ for $p\ge 2$.
\end{lemma}

\begin{proof}[Proof of
proposition~\ref{prop:product-estimate}]
The proof has three steps.
Step~2 requires $p\ge2$.
Abbreviate $I=(a,b]$ and
for $q,r\in[1,\infty]$
consider the norm
$$
     \Norm{\xi}_{q;r}
     :=\Norm{\xi}_{L^q(I,L^r(S^1))}.
$$

\vspace{.1cm}
\noindent
{\sc Step 1.}
{\it 
Fix reals $\alpha\ge1$ and
$q,r,q^\prime,r^\prime\in[\alpha,\infty]$
such that 
$\frac{1}{q}
+\frac{1}{r}
=\frac{1}{\alpha}$
and
$\frac{1}{q^\prime}
+\frac{1}{r^\prime}
=\frac{1}{\alpha}$.
Then
$$
     \Norm{fg}_\alpha
     \le \Norm{f}_{q^\prime;q}
     \Norm{g}_{r^\prime;r}
$$
for all functions
$f,g\in C^\infty(I\times S^1)$.
}

\vspace{.1cm}
\noindent
Let $f_s(t):=f(s,t)$.
Apply H\"older's inequality twice
to obtain
\begin{equation*}
\begin{split}
     \Norm{fg}_{L^\alpha(I\times S^1)}^\alpha
    &=\int_a^b
     \Norm{f_sg_s}_{L^\alpha(S^1)}^\alpha
     \, ds\\
    &\le\int_a^b
     \left(\Norm{f_s}_{L^q(S^1)}
     \Norm{g_s}_{L^r(S^1)}\right)^\alpha\, ds\\
    &=\Norm{uv}_{L^\alpha(I)}^\alpha\\
    &\le\left(\Norm{u}_{L^{q^\prime}(I)}
     \Norm{v}_{L^{r^\prime}(I)}\right)^\alpha
\end{split}
\end{equation*}
where $u(s):=\Norm{f_s}_{L^q(S^1)}$
and $v(s):=\Norm{g_s}_{L^r(S^1)}$.
This proves Step~1.

\vspace{.1cm}
\noindent
{\sc Step 2.}
{\it
Given $p,c_0$, and $u$
as in the hypothesis of the lemma.
Then there is a constant
$c=c(p,c_0)>0$ such that
$$
     \Norm{\Nabla{t}\xi}_{\infty;p} 
     \le c\Norm{\xi}_{\Ww^{1,p}}
$$
for every smooth compactly 
supported vector field
$\xi$ along $u:I\times S^1\to N$.
}

\vspace{.1cm}
\noindent
The proof uses the
generalized Young inequality:
Given reals $a,b,c\ge 0$ and
$1<\alpha,\beta,\gamma<\infty$
such that
$\frac{1}{\alpha}
+\frac{1}{\beta}
+\frac{1}{\gamma}
=1$, then
\begin{equation}\label{eq:generalized-Young}
     abc
     \le \frac{a^\alpha}{\alpha} 
     +\frac{b^\beta}{\beta} 
     +\frac{c^\gamma}{\gamma}.
\end{equation}

\noindent
Abbreviate $\xi(s,t)$ by $\xi$, then
integration by parts shows that
\begin{equation*}
\begin{split}
    &\frac{d}{ds}
     \int_0^1\Abs{\Nabla{t}\xi(s,t)}^p\, dt\\
    &=p\int_0^1\Abs{\Nabla{t}\xi}^{p-2}
     \langle\Nabla{t}\xi,\Nabla{t}\Nabla{s}\xi
     +[\Nabla{s},\Nabla{t}]\xi\rangle\, dt\\
    &=-p\int_0^1\left(\frac{d}{dt}
     \Abs{\Nabla{t}\xi}^{p-2}\right)
     \langle\Nabla{t}\xi,\Nabla{s}\xi\rangle\, dt
     -p\int_0^1\Abs{\Nabla{t}\xi}^{p-2}
     \langle\Nabla{t}\Nabla{t}\xi,
     \Nabla{s}\xi\rangle\, dt\\
    &\quad +p\int_0^1\Abs{\Nabla{t}\xi}^{p-2}
     \langle\Nabla{t}\xi,R(\p_su,\p_tu)\xi\rangle
     \, dt\\
    &=-p(p-2)\int_0^1\Abs{\Nabla{t}\xi}^{p-4}
     \langle\Nabla{t}\xi,
     \Nabla{t}\Nabla{t}\xi\rangle
     \langle\Nabla{t}\xi,\Nabla{s}\xi\rangle\, dt\\
    &\quad -p\int_0^1\Abs{\Nabla{t}\xi}^{p-2}
     \left(
     \langle\Nabla{t}\Nabla{t}\xi,\Nabla{s}\xi\rangle
     -\langle\Nabla{t}\xi,R(\p_su,\p_tu)\xi\rangle\right)
     \, dt.
\end{split}
\end{equation*}
Take the absolute value
of the right hand side,
apply the generalized
Young inequality~(\ref{eq:generalized-Young})
in the case\footnote{
The case $p=2$ is taken care of
by the standard Young inequality.}
$p>2$
with $\alpha=p/(p-2), \beta=p, \gamma=p$,
and the standard Young inequality
with $\alpha=p/(p-1), \beta=p$
to obtain the inequality
\begin{equation*}
\begin{split}
    &\frac{d}{ds}
     \int_0^1\Abs{\Nabla{t}\xi(s,t)}^p\, dt\\
    &\le p(p-1)\int_0^1\Abs{\Nabla{t}\xi}^{p-2}
     \Abs{\Nabla{t}\Nabla{t}\xi}\cdot
     \Abs{\Nabla{s}\xi}\, dt
     +pc_0^2\Norm{R}_\infty
     \int_0^1\Abs{\Nabla{t}\xi}^{p-1}
     \Abs{\xi}\, dt\\
    &\le p(p-1)\int_0^1\left(
     \frac{p-2}{p} \Abs{\Nabla{t}\xi}^p
     +\frac{1}{p}\Abs{\Nabla{t}\Nabla{t}\xi}^p
     +\frac{1}{p}\Abs{\Nabla{s}\xi}^p
     \right)\, dt\\
    &\quad +pc_0^2\Norm{R}_\infty
     \int_0^1\left(\frac{p-1}{p}
     \Abs{\Nabla{t}\xi}^p
     +\frac{1}{p}\Abs{\xi}^p\right)\, dt\\
    &\le C_1\left(\Norm{\xi_s}_{L^p(S^1)}^p
     +\Norm{\Nabla{s}\xi_s}_{L^p(S^1)}^p
     +\Norm{\Nabla{t}\Nabla{t}\xi_s}_{L^p(S^1)}^p
     \right).
\end{split}
\end{equation*}
Here $C_1>0$ is a constant
depending only on
$p,c_0,$ and $\Norm {R}_\infty$
and $\xi_s(t):=\xi(s,t)$.
Note that we used
lemma~\ref{le:plus-minus}
to estimate the terms
involving $\Nabla{t}\xi_s$.
Now fix $\sigma\in(a,b]$ and
integrate this inequality
over $s\in(a,\sigma]$
to obtain the estimate
$$
    \Norm{\Nabla{t}\xi_\sigma}_{L^p(S^1)}^p
    \le c\Norm{\xi}_{\Ww^{1,p}((a,b]\times S^1)}^p.
$$
Here we used compactness
of the support of $\xi$ and monotonicity
of the integral.
Since the right hand side is independent of $\sigma$,
the proof of Step~2 is complete.

\vspace{.1cm}
\noindent
{\sc Step 3.}
{\it We prove the lemma.}

\vspace{.1cm}
\noindent
Define $f(s,t):=\abs{\Nabla{t}\xi(s,t)}$
and $g(s,t):=\abs{\Nabla{t} X(s,t)}$.
By Step~1 with
$\alpha,q,$ and $r^\prime$ equal to $p$
and with $r=q^\prime=\infty$
we have
$$
     \int_a^b
     \int_0^1\left(\Abs{\Nabla{t}\xi(s,t)}
     \Abs{\Nabla{t} X(s,t)}\right)^p\, dtds
     =\Norm{fg}_p^p \\
     \le\Norm{\Nabla{t} \xi}_{\infty;p}^p
     \Norm{\Nabla{t} X}_{p;\infty}^p.
$$
Now apply Step~2 to the first factor.
For the second one
we exploit the fact that, since
the slices $s\times S^1$
of our domain are compact,
there is the
Sobolev embedding
$W^{1,p}(S^1)
\hookrightarrow L^\infty(S^1)$
with constant $\mu=\mu(p)>0$.
It follows that
\begin{equation*}
\begin{split}
     \int_a^b
     \Norm{\Nabla{t} X_s}_{L^\infty(S^1)}^p\,
     ds
    &\le\int_a^b
     \mu^p\Norm{\Nabla{t} X_s}_{W^{1,p}(S^1)}^p
     \, ds \\
   &=\mu^p \int_a^b
     \Norm{\Nabla{t} X_s}_{L^p(S^1)}^p
     +\Norm{\Nabla{t}\Nabla{t} X_s}_{L^p(S^1)}^p
     \, ds.
\end{split}
\end{equation*}
This concludes the proof of
proposition~\ref{prop:product-estimate}.
\end{proof}

\begin{proof}[Proof of 
theorem~\ref{thm:product-estimate}]
Proposition~\ref{prop:product-estimate}
with $N=\R^k$,
$u\equiv const$,
$\xi=v$, and $X=w$.
\end{proof}

\section{The parabolic Weyl lemma}
\label{sec:parabolic-weyl}

The structure of proof of
lemma~\ref{le:parabolic-weyl}
is the following.
First we approximate $u$ via convolution
by a family of smooth solutions $u_\eps$
which converge to $u$ in $L^1$.
The point is that we convolute
over \emph{individual time slices} 
$s\times\R$ for almost all times $s$
using mollifiers defined on $\R$.
(It is also possible to carry over
the proof of the original Weyl lemma
for the Laplacian using mollifiers
supported in $\R^2$. This leads to
restrictions and is explained
in a separate section below.)
On the other hand, given any
integer $k\ge 0$,
standard local $C^k$ estimates for
smooth solutions
of the linear homogeneous
heat equation
in terms of the $L^1$ norm apply;
see~\cite[Sec.~2.3 Thm.~9]{EVANS-PDE}.
They provide $C^k$ bounds
on compact sets in terms of
$\norm{u_\eps}_1$.
Now by Young's convolution inequality
$\norm{u_\eps}_1\le\norm{u}_1$.
Hence these bounds are uniform in $\eps$.
Therefore by Arzela-Ascoli the family $u_\eps$
converges in $C^{k-1}_{loc}(\Omega)$
to a map $v$. Hence $u=v$
by uniqueness of the limit.
As this is true for
every $k$ and, moreover, every point
is contained in a compact subset
of $\Omega$
it follows that $u\in C^\infty(\Omega)$.
Integration by parts then shows that
\begin{equation}\label{eq:heat-linear}
     \p_su-\p_t\p_tu=0
\end{equation}
on the interior of $\Omega$.
Since $u$ is $C^\infty$
on $\Omega$ this identity
continues to hold on $\Omega$.

\begin{proof}
[Proof of lemma~\ref{le:parabolic-weyl}]
Every point of $\Omega$ is contained
in (some translation of)
a parabolic set
$(-r^2,0]\times(-r,r)$
whose closure is
contained in $\Omega$ for some
$r>0$ sufficiently small.
Hence we may assume without loss of generality
that
$$
     \Omega=(-r^2,0]\times(-r,r),\qquad
     u\in L^1(\Omega).
$$
We prove the lemma in nine steps.

1) We introduce appropriate mollifiers:
Fix a smooth function
$\rho:\R\to[0,1]$ which is compactly supported
in the interval $(-1,1)$ and satisfies
$\int_\R\rho=1$.
For $\eps>0$ consider the mollifier
$$
     \rho_\eps(t):=\frac{1}{\eps}\,
     \rho\left(\frac{t}{\eps}\right).
$$
It is compactly supported in 
the interval $(-\eps,\eps)$
and satisfies $\int_\R\rho_\eps=1$.

2) For almost every $s\in\R$ we define
the family $\{\rho_\eps * u_s\}_{\eps>0}
\subset C_0^\infty(\R)$
and calculate the $L^1$ norm
of its derivatives:
Extend $u$ by zero
on $\R^2\setminus\Omega$
and denote the extension again by $u$.
Then $u\in L^1(\R^2)$
and
$$
     u_s:=u(s,\cdot)\in L^1(\R)
$$
for almost every $s\in\R$.
For such $s$ and $\eps>0$ define
$$
     \left(\rho_\eps * u_s\right) (t)
     =\int_\R\rho_\eps(t-\tau)u_s(\tau)\; d\tau.
$$
In this case $\rho_\eps * u_s\in C_0^\infty(\R)$,
$$
     \Norm{\rho_\eps * u_s-u_s}_{L^1(\R)}
     \to 0 \quad\text{as $\eps\to 0$,}
$$
and $\rho_\eps * u_s$ converges to $u_s$,
as $\eps\to 0$,
pointwise almost everywhere on $\R$;
see~\cite[App.~A]{Jost-PDE}.
Moreover, by Young's convolution
inequality we obtain that
$$
     \Norm{\rho_\eps*u_s}_{L^1(\R)}
     \le\Norm{\rho_\eps}_{L^1(\R)}
     \Norm{u_s}_{L^1(\R)}
     =\Norm{u_s}_{L^1(\R)}
$$
and, more generally, that
\begin{equation*}
\begin{split}
     \Norm{\frac{d^k}{dt^k}
     \left(\rho_\eps*u_s\right)}_{L^1(\R)}
     =\Norm{\left(\rho_\eps^{(k)}*u_s\right)}_{L^1(\R)}
    &\le\Norm{\rho_\eps^{(k)}}_{L^1(\R)}
     \Norm{u_s}_{L^1(\R)}\\
    &=\frac{\Norm{\rho^{(k)}}_{L^1(\R)}}{\eps^k}
     \Norm{u_s}_{L^1(\R)}
\end{split}
\end{equation*}
for every positive integer $k$.
Here $\rho^{(k)}$ denotes
the $k$-th derivative of $\rho$.

3) We prove that for $\eps>0$
the function defined by
$$
     u_\eps:\R^2\to\R,\quad
     (s,t)\mapsto
     (\rho_\eps*u_s)(t)
$$
is integrable and $u_\eps$ converges
to $u$ in $L^1(\R^2)$ as $\eps\to 0$.
Indeed by step~2)
$$
     \Norm{u_\eps}_{L^1(\R^2)}
     =\int_{\R}\Norm{\rho_\eps*u_s}_{L^1(\R)} ds
     \le\int_{\R}\Norm{u_s}_{L^1(\R)} ds
     =\Norm{u}_{L^1(\Omega)}.
$$
Now define the family of functions
$\{f_\eps:\R\to\R\}_{\eps>0}$
for almost every $s$ by
$$
     f_\eps(s)
     :=\Norm{\rho_\eps*u_s-u_s}_{L^1(\R)}.
$$
By the last estimate
these functions are integrable
$$
     \Norm{f_\eps}_{L^1(\R)}
     =\Norm{u_\eps-u}_{L^1(\R^2)}
     \le 2\Norm{u}_{L^1(\Omega)}.
$$
Moreover, they are dominated
almost everywhere
by an integrable function $g$.
Namely, by step~2
$$
     \Abs{f_\eps(s)}\le 2\Norm{u_s}_{L^1(\R)}=:g(s)
     ,\qquad
     \Norm{g}_{L^1(\R)}=2\Norm{u}_{L^2(\Omega)}.
$$
Step~2) again shows that
$f_\eps\to 0$ as $\eps\to 0$
for almost every $s$.
Hence by the Dominated Convergence
Theorem it follows that
\begin{equation*}
\begin{split}
     \lim_{\eps\to 0}
     \Norm{u_\eps-u}_{L^1(\R^2)}
    &=\lim_{\eps\to 0}
     \int_{\R}\Norm{\rho_\eps*u_s-u_s}_{L^1(\R)}\, ds\\
    &=\int_{\R}\left(\lim_{\eps\to 0}f_\eps\right)(s)
     \,ds\\
    &=0.
\end{split}
\end{equation*}

4) The function $u_\eps:\R^2\to\R$
defined in~3) admits integrable
weak $t$-derivatives of all orders:
Fix $\eps>0$
and a positive integer $k$,
then
\begin{equation*}
\begin{split}
     \int_{\R^2} u_\eps\,\p_t^k\psi\,dt\,ds
    &=\int_{\R^2}
     \left(\rho_\eps*u_s\right)
     \p_t^k\psi\,dt\,ds\\
    &=(-1)^k\int_{\R^2}
     (\rho_\eps^{(k)}*u_s)\,
     \psi\,dt\,ds
\end{split}
\end{equation*}
for every $\psi\in C_0^\infty(\R^2)$.
Here $\rho_\eps^{(k)}$ denotes
the $k$-th derivative.
Moreover, the first step
is by definition of $u_\eps$
and the second step
by integration by parts followed by
commuting differentiation and convolution.
Next observe that the function
$\rho_\eps^{(k)}*u_s$
is integrable. Indeed step~2) shows that
$$
     \int_{\R}\bigl\|
     \rho_\eps^{(k)}*u_s\bigr\|_{L^1(\R)} ds
     \le\frac{c_k}{\eps^k}
     \Norm{u}_{L^1(\Omega)}
$$
with constant
$c_k=c_k(\rho)=\norm{\p_t^k\rho}_{L^1(\R)}$.
Hence the weak $t$ derivatives of
the function $u_\eps:\R^2\to\R$
are integrable and given by
$$
     \p_t^ku_\eps(s,t)
     =(\rho_\eps^{(k)}*u_s)(t).
$$

5) Fix $\eps>0$ and consider
the subset
$$
     \Omega_\eps
     =(-r^2,0]\times(-r+\eps,r-\eps)
     \subset\Omega.
$$
We prove by induction that
for every integer $k\ge1$
the weak derivative $\p_s^ku_\eps$
exists in $L^1(\Omega_\eps)$
and equals $\p_t^{2k}u_\eps$
almost everywhere
on $\Omega_\eps$.
Here assumption~(\ref{eq:weak-par}) enters.

\noindent
{\it Case $k=1$.}
Straightforward calculation shows that
\begin{equation*}
\begin{split}
     \int_\Omega \psi\, \p_t\p_t u_\eps
    &=\int_{\R^2} \psi(s,t)\left(
     \int_\R 
     \p_t\p_t\rho_\eps(t-\tau) u_s(\tau)
     \, d\tau \right) ds\,dt\\
    &=\int_{\R^3} \psi(s,t)\, u(s,\tau)\,
     \p_\tau\p_\tau\rho_\eps(t-\tau)\,
     d\tau\, ds\, dt\\
    &=\int_{\R} \left(\int_{\R^2} u(s,\tau)\,
     \p_\tau\p_\tau\Bigl(\rho_\eps(t-\tau)
     \psi(s,t)\Bigr)\, d\tau\,ds\right)
     dt\\
    &=-\int_{\R} \left(\int_{\R^2} u(s,\tau)\,
     \p_s\Bigl(\rho_\eps(t-\tau)
     \psi(s,t)\Bigr)\, d\tau\,ds\right)
     dt\\
    &=-\int_{\R^2} \left(
     \int_\R \rho_\eps(t-\tau) u_s(\tau)\, d\tau
     \right) \p_s\psi(s,t)\,ds\,dt\\
    &=-\int_\Omega u_\eps \p_s\psi
\end{split}
\end{equation*}
for every test function
$\psi\in C_0^\infty(\INT\,\Omega_\eps)$.
This identity means that on $\INT\,\Omega_\eps$,
hence on $\Omega_\eps$, the weak
derivative $\p_su_\eps$
exists and equals
$\p_t\p_tu_\eps$ which
is integrable by~4).
To prove the identity
note that the first and the final step
are by definition of $u_\eps$ in~3).
To obtain the second step we
changed the order of integration
and applied the chain rule.
Steps three and five are obvious.
To obtain step four we used
assumption~(\ref{eq:weak-par})
and the fact that
$$
     \phi_t(s,\tau)
     :=\rho_\eps(t-\tau)\psi(s,t)
$$
lies in $C^\infty_0(\INT\,\Omega)$
for every $t\in\R$.
To prove this assume that
$\phi_t(s,\tau)\not=0$.
This means firstly that
$\rho_\eps(t-\tau)\not=0$,
hence $\tau\in[-\eps+t,\eps+t]$,
and secondly that $\psi(s,t)\not=0$.
Now fix a sufficiently small
constant $\delta=\delta(\eps)>0$ such that
$$
     \supp\,\psi
     \subset
     [-r^2+\delta,-\delta]
     \times [-r+\eps+\delta,r-\eps-\delta]
     \subset
     \INT\,\Omega_\eps.
$$
It follows that
\begin{equation*}
\begin{split}
     (s,\tau)
    &\in[-r^2+\delta,-\delta]\times
     [-\eps+(-r+\eps+\delta),\eps+(r-\eps-\delta)]\\
    &=[-r^2+\delta,-\delta]\times
     [-r+\delta,r-\delta]
     \subset\INT\,\Omega.
\end{split}
\end{equation*}

\noindent
{\it Induction step $k\Rightarrow k+1$.}
The calculation
follows the same steps as above.
We only indicate the minor differences.
Assume that case $k$ is true, then
\begin{equation*}
\begin{split}
     \int_\Omega \psi\, \p_t^{2k+2} u_\eps
    &=(-1)^{k+1}\int_{\R} \left(\int_{\R^2} u(s,\tau)\,
     \p_s^{k+1}\Bigl(\rho_\eps(t-\tau)
     \psi(s,t)\Bigr)\, d\tau\,ds\right)
     dt\\
    &=(-1)^{k+1}\int_{\R^2} u_\eps(s,t)\,
     \p_s^{k+1} \psi(s,t)\,ds\,dt\\
    &=-\int_\Omega \left(\p_s^ku_\eps\right)
     \p_s\psi
\end{split}
\end{equation*}
for every test function
$\psi\in C_0^\infty(\INT\,\Omega_\eps)$.
Note that to obtain
the first step
we applied $k+1$ times
assumption~(\ref{eq:weak-par})
using that $\phi_t$ and therefore also its
derivatives are in $C_0^\infty(\INT\,\Omega)$.
In the final step
we used the induction
hypothesis to integrate
by parts $k$ times
the $s$ variable.

6) The function
$u_\eps$ is smooth
on the closure of $\Omega_\eps$:
Fix $\eps>0$ and positive integers $m$ and $\ell$.
Then $\p_t^m\p_s^\ell u_\eps$
equals $\p_t^{m+2\ell} u_\eps$
almost everywhere on $\Omega_\eps$ by~5)
and the latter function is integrable by~4).
This proves that
$$
     u_\eps\in
     \bigcap_{k=1}^\infty W^{k,1}(\Omega_\eps)
     =C^\infty(\overline\Omega_\eps).
$$
Moreover, by~5) with $k=1$,
each $u_\eps$ solves
the linear heat
equation~(\ref{eq:heat-linear})
on $\overline\Omega_\eps$.

7) From now on fix a compact
subset $Q\subset\Omega$.
Then for every positive integer $k$
the family $u_\eps$ is uniformly bounded
in the Banach space
$C^k(Q)$ by a constant
$\mu_k=\mu_k(Q)$:
To see this consider
the compact parabolic set
of radius $r$, height $r^2$,
and top center point $(s,t)\in Q$
given by
$$
     P_r(s,t)
     :=[s-r^2,s]\times[t-r,t+r].
$$
By compactness of $Q$
there is a constant
$\eps_0=\eps_0(Q)>0$
such that $Q\subset \Omega_{\eps_0}$
and, moreover, there is a constant
$\rho=\rho(\eps_0,Q)>0$ such that
$$
     P_{2\rho}(s,t)\subset\Omega_{\eps_0}
$$
for every point $(s,t)\in Q$.
By step~6) each function $u_\eps$
with $\eps\in(0,\eps_0)$
is a smooth solution
of the linear homogeneous
heat equation~(\ref{eq:heat-linear})
on the domain $\Omega_\eps$
and therefore on
$\Omega_{\eps_0}$.
Now given a point $(\sigma,\tau)\in Q$
and a pair of nonnegative
integers $m,\ell$ there is
by~\cite[Sec.~2.3 Thm.~9]{EVANS-PDE}
a constant $c_{m,\ell}(\sigma,\tau)$
such that
$$
     \max_{P_\frac{\rho}{2}(\sigma,\tau)}
     \Abs{\p_t^m\p_s^\ell v}
     \le \frac{c_{m,\ell}(\sigma,\tau)}{\rho^{m+2\ell+3}}
     \Norm{v}_{L^1(P_\rho(\sigma,\tau))}
$$
for all smooth solutions $v$ of
the heat equation~(\ref{eq:heat-linear})
in $P_{2\rho}(\sigma,\tau)$.
By compactness of $Q$ there are finitely many
sets $P_{\rho/2}(\sigma_\nu,\tau_\nu)$ covering $Q$.
Then the corresponding estimates 
for $v=u_\eps$ and
$m,\ell=0,1,\ldots,k$ imply that
$$
     \Norm{u_\eps}_{C^k(Q)}
     \le \alpha \Norm{u_\eps}_{L^1(\R^2)}
     \le \alpha \Norm{u}_{L^1(\Omega)}
$$
for every $\eps\in(0,\eps_0)$
and where the constant $\alpha>0$
depends only on the compact set $Q$
(since $\rho$ eventually depends
on $Q$ only).
Here the second inequality is proved in step~3).

8) We prove that $u\in C^\infty(Q)$.
In the setting of step~7) the Arzela-Ascoli
theorem for each $k$
together with choosing a diagonal
subsequence yields existence of
a sequence $\eps_k\to 0$, as $k\to\infty$,
and a smooth function $\hat u$ defined on $Q$
such that $u_{\eps_k}\to \hat u$ 
in $C^\infty(Q)$, as $k\to\infty$.
On the other hand, the sequence
$u_{\eps_k}$ converges to $u$
in $L^1(Q)$ by step~3).
Hence $u=\hat u$ by uniqueness of limits.

9) We prove lemma~\ref{le:parabolic-weyl}.
Since every point of $\Omega$
is contained in a compact subset $Q$
and $u\in C^\infty(Q)$ by step~8),
the function $u$ is smooth on $\Omega$.
To prove the identity $\p_su-\p_t\p_tu=0$
on $\Omega$
assume by contradiction that
this identity is violated
at a point $(s_*,t_*)$ of $\Omega$.
There are two cases.
\\
If $(s_*,t_*)$ is in the interior of $\Omega$,
then by smoothness of $u$
there is a sufficiently small 
open neighborhood $U$ of $(s_*,t_*)$ in $\Omega$
and a function $\phi\in C^\infty_0(U,[0,1])$
with $\phi(s_*,t_*)=1$
such that assumption~(\ref{eq:weak-par}) fails.
(For instance, if $c>0$ 
is the value of the function
$\p_su-\p_t\p_tu$ at the point $(s_*,t_*)$,
let $U$ be the subset of $\Omega$
on which $\p_su-\p_t\p_tu>c/2$.)
\\
If $(s_*,t_*)$ is in the boundary
$0\times(-r,r)$ of $\Omega$, the former
argument works for an interior point
of $\Omega$ sufficiently 
close to $(s_*,t_*)$.
Existence of such an interior point uses again
smoothness of $u$ on $\Omega$.
This proves the parabolic Weyl lemma.
\end{proof}

\subsection*{The heat ball approach}
\label{subsec:heat-ball-approach}

A natural first try to prove
lemma~\ref{le:parabolic-weyl}
is to carry over the proof
of the original Weyl lemma
for the Laplacian;
see e.g.~\cite{GT77,Jost-PDE}).
This works, but with two restrictions.
Firstly, the set $\Omega$ should be open
in $\R^2$ and, secondly,
the function $u$ should be locally $L^q$
integrable over $\Omega$ for some $q>3$.

The original proof is based on the fact
that harmonic functions are characterized
by their mean value property
with respect to balls or spheres.
There is a similar statement
for solutions to the heat equation.
However, in the corresponding
parabolic mean value equalities
a weight factor different from one appears and
this eventually leads to the restriction $q>3$.
A further difference is that
balls and spheres over which the means
are taken are replaced by heat balls
and their boundaries, respectively.
The parabolic mean value property
with respect to boundaries
is due to Fulks~\cite{Fu66} and with respect to
heat balls it is due to Watson~\cite{Wa73}.
Here it is required that $\Omega$
is open in $\R^2$.

Recall that the
\emph{fundamental solution to the heat equation}
is given by
\begin{equation}\label{eq:fundamental-solution}
    \Phi(s,t):=
  \begin{cases}\displaystyle
    \frac{1}
    {\sqrt{4\pi s}}\;
    e^{\textstyle -\frac{\textstyle t^2}{\textstyle 4s}}
    &\text{, $s>0$, $t\in\R$,}\\
    0&\text{, $s<0$, $t\in\R$.}
  \end{cases}
\end{equation}
For $r>0$ we denote by $E_r=E_r(0,0)$
the area which is enclosed by the level set
determined by the identity
$$
    \Phi(-s,-t)
    =\frac{1}{2r\sqrt{\pi}}.
$$
This level set
is parametrized by
$$
     t(s)=\pm\sqrt{2s\ln\frac{-s}{r^2}},\quad
     s\in(-r^2,0).
$$
Think of it
as resembling an ellipse in the plane
such that the origin is located
at the 'north pole'.
For general base point $(s,t)\in\R^2$
the set $E_r(s,t)$
is defined by translation.
These sets are called
\emph{heat balls of ``radius'' $r$}.
Following Watson~\cite{Wa73}
we call a function $u$ defined
on an open subset $\Omega\subset\R^2$
a \emph{temperature} if
$\p_t\p_tu$ and $\p_su$
are continuous functions on $\Omega$
and the heat equation
$\p_su-\p_t\p_tu=0$
is satisfied pointwise
on $\Omega$.
(Note that temperatures are automatically
$C^\infty$ smooth;
see e.g.~\cite[Sec.~2.3 Thm.~8]{EVANS-PDE}.)

\begin{theorem}[\cite{Wa73} \S~10 Cor.~1]
\label{thm:watson}
Let $u$ be a continuous function
on an open subset $\Omega\subset\R^2$.
Then the following are equivalent.
\begin{enumerate}
\item[\rm(a)]
  The function $u$ is a temperature.
\item[\rm(b)]
  At every point $(s,t)\in\Omega$
  the weighted mean value equality for $u$
  holds
  $$
     u(s,t)
     =\frac{1}{8\sqrt{\pi}\cdot r}
     \int_{E_r(s,t)} 
     \frac{(t-\tau)^2}{(s-\sigma)^2}\; 
     u(\sigma,\tau)\; d\tau\, d\sigma
  $$
  whenever
  $\overline{E_r(s,t)}\subset\Omega$.
\end{enumerate}
\end{theorem}

We sketch the proof of the parabolic Weyl lemma 
(subject to the two restrictions mentioned above)
along the lines of the original proof
for the Laplacian.
Since smoothness is a local property
we may assume without loss of generality
that $\Omega\subset\R^2$ is bounded. Moreover, we extend
$u$ by zero to $\R^2\setminus\Omega$
without change of notation.
Hence $u\in L^q(\Omega)$ for some $q>3$.
The main idea is to mollify the given weak solution
$u$ to obtain a family
$\{u_r\}\subset C^\infty_0(\R^2)$
of smooth functions converging in $L^1$,
hence almost everywhere, to $u$.
Here we use a family of mollifiers
$\{\rho_r\}$ which are compactly supported
in the heat ball $E_r\subset\R^2$
and set $u_r=\rho_r*u$ where $*$ denotes convolution.
Assumption~(\ref{eq:weak-par})
is then used to show that each function $u_r$
is a temperature on a slightly smaller
set $\Omega_r\subset\Omega$ which by definition
consists of all points $(s,t)\in\Omega$
such that the closure of
the heat ball $E_r(s,t)$
is contained in $\Omega$.
Hence each $u_r:\Omega_r\to\R$ satisfies the
weighted mean value equality
of theorem~\ref{thm:watson}.
On the other hand, the family $\{u_r\}$
is uniformly bounded --
here the restriction $q>3$ arises --
and equicontinuous.
Hence by Arzela-Ascoli it converges
in $C^0$ to a continuous
function $v$ as $r\to 0$.
Since the functions $u_r$
satisfy the mean value equality,
so does their $C^0$ limit $v$,
and therefore $v$ is a temperature
by Watson's result theorem~\ref{thm:watson}.
But $v=u$, since $\{u_r\}$
converges to $u$ almost everywhere.

As it is essentially the only point where
the proof of the original Weyl lemma
for the Laplacian differs
we provide the details
of uniform boundedness of
the family $\{u_r\}$ on $\Omega_R$.
More precisely, fix a constant $R>0$
and restrict $r$ to the interval $(0,R/2)$.
Then
$$
     \Omega_R\subset
     \Omega_{R/2}
     \subset\Omega_r
     \subset\Omega,
     \qquad
     \overline{E_{R/2}(s,t)}
     \subset\Omega_{R/2}\quad
     \forall (s,t)\in\Omega_R.
$$
Hence by theorem~\ref{thm:watson}
each temperature $u_r$
satisfies the mean value equality
on all heat balls with base point
in $\Omega_R$
and radius less or equal to $R/2$.
To see that the family
$\{u_r\}_{r\in(0,R/2)}$ is
uniformly bounded on $\Omega_R$
fix a point $(s_0,t_0)\in\Omega_R$.
Then by the mean value
equality for the temperature $u_r$ over
the heat ball $E_{R/2}(s_0,t_0)$
it follows that
\begin{equation*}
\begin{split}
     \Abs{u_r(s_0,t_0)}
    &\le\frac{1}{4\sqrt{\pi} R}
     \int_{E_{R/2}(s_0,t_0)}
     \frac{(t_0-\tau)^2}{(s_0-\sigma)^2}
     \Abs{u_r(\sigma,\tau)}
     \; d\tau d\sigma\\
    &=\frac{1}{4\sqrt{\pi} R}
     \int_{E_{R/2}(0,0)}
     \frac{t^2}{s^2}
     \Abs{u_r(s+s_0,t+t_0)}
     \; dt ds\\
    &\le\frac{1}{4\sqrt{\pi} R}
     \Norm{t^2s^{-2}}_{L^p(E_{R/2})}
     \Norm{u_r}_{L^q(\R^2)}\\
    &\le c_{q,R}\Norm{u}_{L^q(\Omega)}.
\end{split}
\end{equation*}
To obtain step two we
introduced new variables
$t=\tau-t_0$ and $s=\sigma-s_0$.
In step three we use
H\"older's inequality with $1/p+1/q=1$
and $p,q>1$.
Since the weight function $t^2s^{-2}$
is not bounded on $E_{R/2}$
we can't get away with
pulling out the sup norm
as in the proof of the original
Weyl lemma for the Laplacian
where the weight is one.
In the last step we used that
$$
     \Norm{u_r}_{L^q(\R^2)}
     =\Norm{\rho_r*u}_{L^q(\R^2)}
     \le\Norm{\rho_r}_{L^1(\R^2)}\Norm{u}_{L^q(\R^2)}
     =\Norm{u}_{L^q(\Omega)}
$$
by Young's convolution inequality.
Moreover, the constant $c_{q,R}$ is given by
$\norm{t^2s^{-2}}_{L^p(E_{R/2})}/
4\sqrt{\pi} R$ with $p=\frac{q}{q-1}$.
To see that it is finite
observe that
\begin{equation*}
\begin{split}
     \Norm{t^2s^{-2}}_{L^p(E_1)}^p
    &=\frac{2^{p+\frac{3}{2}}}{2p+1}
     \int_{-1}^0
     \frac{\left(s\ln(-s)\right)^{p+\frac{1}{2}}}
     {(-s)^{2p}} ds\\
    &=\frac{2^{p+\frac{3}{2}}}{2p+1}
     \int_0^\infty
     x^{p+\frac{1}{2}} e^{-x(\frac{3}{2}-p)} dx\\
    &=\frac{2^{p+\frac{3}{2}}}{2p+1}
     \frac{\Gamma(p+\frac{3}{2})}
     {\left(\frac{3}{2}-p\right)^{p+\frac{3}{2}}}.
\end{split}
\end{equation*}
Here we used the change of
variables $x=-\log(-s)$ in the second step,
the last step is valid whenever
$-\frac{3}{2}<p<\frac{3}{2}$, and $\Gamma$
denotes the gamma function.
The earlier use of H\"older's inequality
further restricts $p$ to the
interval $(1,\frac{3}{2})$
and this is equivalent
to $q=\frac{p}{p-1}>3$.
It remains to replace the unit heat ball $E_1$
by $E_{R/2}$. This leads to a further constant
which depends only on $R$ and $p$.

\section{Local regularity}
\label{sec:local-regularity}

The parabolic analogue of
the Calderon-Zygmund
inequality is the following
fundamental $L^p$ estimate.
It is used in the proof of
theorem~\ref{thm:local-regularity}
on local regularity
and it implies the interior estimates
of theorem~\ref{thm:interior-estimate}
by induction.

\begin{theorem}[Fundamental $L^p$ estimate]
\label{thm:kerD-CalZyg}
For every $p>1$,
there is a constant 
$c=c(p)>0$ such that
$$
     \norm{\p_sv}_p
     +\norm{\p_t\p_tv}_p
     \le c
     \norm{\p_sv-\p_t\p_tv}_p
$$
for every $v\in C_0^\infty(\R^2)$.
The same statement is even true for
the domain $\HH^-$.
\end{theorem}

\begin{proof}
A proof for $\R^2$
is given in~\cite[theorem~C.2]{SaJoa-LOOP}
by the Marcinkiewicz-Mihlin
multiplier method.
In the case of the lower half plane $\HH^-$ 
choose a compactly supported
smooth function $v$ on $\HH^-$
and constants $T>0$ and $a<b$
such that 
$\supp\,u\subset(-T/2,0]\times(a,b)$.
Then~\cite[proposition~7.11]{LIEBERMAN}
with $n=1$, $A^{11}=1$, 
$\lambda=\Lambda=1$,
the cube 
$K_0=(-T/2,0]\times(a,b)$
in $(-T,0)\times \R$, and
the function 
$f=\p_su-\p_t\p_tu$ proves the statement.
Note that the case $\HH^-$ implies the case $\R^2$
by translation.
\end{proof}

\begin{theorem}[Local regularity]
\label{thm:local-regularity}
Fix a constant $1<q<\infty$,
an integer $k\ge0$, and an open
subset $\Omega\subset\HH^-$.
Then the following is true.
\begin{itemize}

\item[\rm a)]
If $u\in L^1_{loc}(\Omega)$ and
$f\in\Ww^{k,q}_{loc}(\Omega)$ satisfy
\begin{equation}\label{eq:weak}
     \int_{\Omega}
     u\left(-\p_s\phi-\p_t\p_t\phi\right)
     =\int_{\Omega} f\phi
\end{equation}
for every
$\phi\in C_0^\infty(\INT\,\Omega)$,
then $u\in\Ww^{k+1,q}_{loc}(\Omega)$.

\item[\rm b)]
If $u\in L^1_{loc}(\Omega)$ and
$f,h\in\Ww^{k,q}_{loc}(\Omega)$ satisfy
\begin{equation}\label{eq:weak-divergence}
     \int_{\Omega}
     u\left(-\p_s\phi-\p_t\p_t\phi\right)
     =\int_{\Omega} f\phi
     -\int_{\Omega} h\,\p_t\phi
\end{equation}
for every $\phi\in C_0^\infty(\INT\,\Omega)$,
then $u$ and $\p_tu$ are in
$\Ww^{k,q}_{loc}(\Omega)$.
\end{itemize}
\end{theorem}

Here $\INT\,\Omega$ denotes the interior of $\Omega$.
While part~b) is not needed in this text
it is used in~\cite{Joa-HEATFLOW}
to prove regularity of the solutions
of the \emph{linearized} heat equation.
For convenience of the reader
we recall Poincar\'{e}'s inequality and its proof.
It is used to prove
theorem~\ref{thm:local-regularity}
and theorem~\ref{thm:interior-estimate}.

\begin{lemma}[Poincar\'{e}'s inequality]
\label{le:poincare-ineq}
Fix constants $q\ge 1$ and $r>0$.
Then
$$
     \Norm{\varphi}_q\le 2r \Norm{\p_t\varphi}_q
$$
for every $\varphi\in C^\infty_0((-r,r))$.
\end{lemma}

\begin{proof}
For such $\varphi$ it holds that
$\varphi(-r)=0$ and hence
$
     \varphi(t)
     =\int_{-r}^t \p_t\varphi(\tau) \,d\tau
$
by the  fundamental theorem of calculus.
This implies that 
$$
     \Abs{\varphi(t)}
     \le \int_{-r}^t \Abs{\p_t\varphi(\tau)} \,d\tau
     \le \int_{-r}^r 1\cdot\Abs{\p_t\varphi(\tau)}
         \,d\tau
     \le (2r)^{1/p}\Norm{\p_t\varphi}_q
$$
where the last step uses H\"older's inequality
with $1/q+1/p=1$. Therefore
$$
     \Abs{\varphi(t)}^q
     \le (2r)^{q-1} \Norm{\p_t\varphi}_q^q
$$
and integration over $t\in(-r,r)$
concludes the proof of the lemma.
\end{proof}

\begin{proof}[Proof of 
theorem~\ref{thm:local-regularity} ]
Since any given compact subset $Q$
of $\Omega$ can be covered by
finitely many parabolic rectangles
whose closure is contained in $\Omega$,
we may assume without loss of generality
that $\Omega=(-r^2,0]\times(-r,r)$
for $r>0$.
\\
{\it ad a)} The proof consists of four steps.

I) Fix two open subsets 
$\Omega^\prime$ and $U$ of
$\Omega=(-r^2,0]\times(-r,r)$
such that the closure of $\Omega^\prime$
is contained in $U$
and the closure of $U$
is contained in $\Omega$.
Fix a smooth compactly
supported cutoff function
$\beta:\Omega\to[0,1]$ such that
$\beta=1$ on $U$.
Then $\beta f$ is compactly supported
and $\Ww^{k,q}$ integrable over $\Omega$.
Now approximate $\beta f$ 
in $\Ww^{k,q}(\Omega)$
through a sequence 
$(f_i)\subset C_0^\infty(\Omega)$, i.e.
$$
     \Norm{f_i-\beta f}_{\Ww^{k,q}(\Omega)}
     \longrightarrow 0, \qquad
     \text{as $i\to\infty$.}
$$

II) Each smooth problem
\begin{equation}\label{eq:smooth-problem}
     (\p_s-\p_t\p_t)u_i=f_i
\end{equation}
with 
$f_i\in  C_0^\infty(\Omega)$
admits a unique solution
$u_i\in C_0^\infty(\Omega)$;
see e.g.~\cite[Thm.~5.6]{LIEBERMAN}.
We prove below that the sequence of solutions $u_i$
is a Cauchy sequence in $\Ww^{k+1,q}(\Omega)$.
Therefore it admits a unique limit
$\hat u \in\Ww^{k+1,q}(\Omega)$.
Now the limit $\hat u$ solves
the identity
$(\p_s-\p_t\p_t)\hat u=\beta f$ 
almost everywhere on $\Omega$
as can be seen as follows:
The sequence
$\p_su_i-\p_t\p_tu_i$ converges to
$\p_s\hat u-\p_t\p_t\hat u$
in $L^q$, since $u_i$
is a Cauchy sequence in $\Ww^{k+1,q}(\Omega)$,
and the sequence
$f_i$ converges to $\beta f$ by step~I).
Uniqueness of the limit then
proves equality in $L^q(\Omega)$.
\\
It remains to prove
that the sequence $u_i$ is Cauchy. 
All norms are with respect to the domain $\Omega$.
Note that
$$
     \Norm{u_i-u_j}_q
     \le 2r \Norm{\p_t(u_i-u_j)}_q
     \le (2r)^2 
         \Norm{\p_t\p_t(u_i-u_j)}_q
$$
Here the first inequality follows
by integrating
Poincar\'{e}'s inequality
(lemma~\ref{le:poincare-ineq})
for $\varphi(t)=u_i(s,t)-u_j(s,t)$
over $s\in(-r^2,0)$.
The second inequality follows similarly.
Now use equation~(\ref{eq:smooth-problem})
to obtain that
$$
     \Norm{u_i-u_j}_q
     \le (2r)^2 \left(
     \Norm{\p_s(u_i-u_j)}_q
     +\Norm{f_i-f_j}_q
     \right).
$$
More generally, there is a constant
$C=C(k,r)$ such that
$$
     \Norm{u_i-u_j}_{\Ww^{k+1,q}}
     \le C \left(
     \Norm{\p_s^{k+1}(u_i-u_j)}_q
     +\Norm{f_i-f_j}_{\Ww^{k,q}}
     \right).
$$
for all $i$ and $j$.
This follows by inspecting
the left hand side term by term
replacing any two $t$-derivatives
by one $s$-derivative and the error term $f_i$
according to equation~(\ref{eq:smooth-problem}).
If an odd number of $t$-derivatives
appears then use lemma~\ref{le:poincare-ineq}
to obtain an even number.
Now the fundamental $L^p$ estimate
theorem~\ref{thm:kerD-CalZyg}
with constant $c=c(q)$
and function $v=\p_s^k(u_i-u_j)$
asserts that
\begin{equation*}
\begin{split}
     \norm{\p_s^{k+1}(u_i-u_j)}_q
    &\le c \norm{(\p_s-\p_t\p_t)\p_s^k(u_i-u_j)}_q\\
    &= c \norm{\p_s^k(f_i-f_j)}_q\\
    &\le c \norm{f_i-f_j}_{\Ww^{k,q}}.
\end{split}
\end{equation*}
Here we used again equation~(\ref{eq:smooth-problem}).
Next use the approximation of $\beta f$ in
step~I) to obtain that the sequence
$u_i$ in $\Ww^{k,q}(\Omega)$
is Cauchy, namely
$$
     \norm{f_i-f_j}_{\Ww^{k,q}}
     \le \norm{f_i-\beta f}_{\Ww^{k,q}}
     +\norm{\beta f-f_j}_{\Ww^{k,q}}
     \longrightarrow 0,
     \qquad \text{as $i,j\to\infty$}.
$$

III) The restriction of
$\hat u - u$ to the open subset $U\subset \Omega$
is a weak solution of the
homogeneous problem. More precisely, it is true that
\begin{equation*}
\begin{split}
     \int_U(\hat u-u)
     (-\p_s\phi-\p_t\p_t\phi)
    &=\int_U(\p_s\hat u-\p_t\p_t\hat u)\phi
     -\int_U u(-\p_s\phi-\p_t\p_t\phi)\\
    &=\int_U(\p_s\hat u-\p_t\p_t\hat u-\beta f)\phi\\
    &=0
\end{split}
\end{equation*}
for every test function 
$\phi\in C_0^\infty(\INT\,U)$.
Here the first step is by
integration by parts using step~II)
and the second step is by assumption~(\ref{eq:weak})
and the fact that $f=\beta f$ on $U$.
The last step uses the identity in step~II).

IV) The difference $\hat u-u$ is in $L^1(U)$
by step~II) and assumption on $u$.
Hence by the parabolic Weyl 
lemma~\ref{le:parabolic-weyl}
the function $F:=\hat u-u$ is smooth
on $U$. Together with the fact that
$\hat u \in\Ww^{k+1,q}(\Omega)$
proved in step~II) this shows that
$u=\hat u-F$ is of class
$\Ww^{k+1,q}$ on each
bounded open subset of $U$,
in particular on $\Omega^\prime$.
This proves part~a) of
theorem~\ref{thm:local-regularity}.

\vspace{.1cm}
\noindent
{\it ad b)} The proof takes four further steps.

V) Let the sets $\Omega^\prime$ and $U$,
the cutoff function $\beta$,
and the sequence $(f_i)\subset C^\infty_0(\Omega)$
be as in step~I).
Approximate the compactly supported function
$\beta h$ in $\Ww^{k,q}(\Omega)$
through a sequence $(h_i)\subset C^\infty_0(\Omega)$.
Now as in steps~II) and~III)
each smooth problem
\begin{equation}\label{eq:smooth-problem2}
     (\p_s-\p_t\p_t)v_i=h_i
\end{equation}
admits a unique solution 
$v_i\in C_0^\infty(\Omega)$
and the sequence $(v_i)$ is Cauchy
in $\Ww^{k+1,q}(\Omega)$ with unique
limit $\hat v$ which solves
the identity $(\p_s-\p_t\p_t)\hat v=\beta h$
almost everywhere on $\Omega$.

VI) Observe that the
sequences 
$$
     w_i:=u_i+\p_t v_i,\qquad
     \p_tw_i=\p_tu_i+\p_t\p_t v_i,
$$
converge in $\Ww^{k,q}(\Omega)$
to the limits 
$$
     \hat w=\hat u+\p_t\hat v,\qquad
     \p_t\hat w=\p_t\hat u+\p_t\p_t\hat v,
$$
respectively.
Moreover, each $w_i$
satisfies the identity
$(\p_s-\p_t\p_t)w_i=f_i+\p_th_i$ on $\Omega$.
Integration by parts then shows that
$$
     \int_\Omega w_i\left(-\p_s-\p_t\p_t\right) \phi
     =\int_\Omega f_i \phi
     -\int_\Omega h_i \p_t\phi
$$
for every $\phi\in C^\infty_0(\INT\,\Omega)$.
Taking the limit $i\to\infty$
we obtain that
\begin{equation}\label{eq:hat-w}
     \int_\Omega \hat w \left(-\p_s-\p_t\p_t\right) \phi
     =\int_\Omega \beta f \phi
     -\int_\Omega \beta h \,\p_t\phi
\end{equation}
for every $\phi\in C^\infty_0(\INT\,\Omega)$.

VII) The restriction
of $\hat w-u$ to the open subset $U$ of $\Omega$
is a weak solution of the homogeneous problem,
meaning that
\begin{equation*}
\begin{split}
     \int_U(\hat w-u)
     (-\p_s\phi-\p_t\p_t\phi)
    &=\int_U\hat w (-\p_s-\p_t\p_t)\phi
     -\int_U u(-\p_s\phi-\p_t\p_t\phi)\\
    &=\int_U\left(\beta f\phi-\beta h\,\p_t\phi\right)
      -\int_U\left(f\phi-h\,\p_t\phi\right)\\
    &=0
\end{split}
\end{equation*}
for every test function 
$\phi\in C_0^\infty(\INT\,U)$.
Here step two uses the
identity~(\ref{eq:hat-w}) for $\hat w$
and assumption~(\ref{eq:weak-divergence}) on $u$.
Step three is true since
$\beta= 1$ on $U$.

VIII) Note that the difference
$\hat w-u$ is in $L^1(U)$ by step~VI)
and assumption on $u$.
Hence by the parabolic Weyl 
lemma~\ref{le:parabolic-weyl}
the function $G:=\hat w-u$ is smooth
on $U$. Since $\hat w \in\Ww^{k,q}(\Omega)$
by step~VI), this shows that
$u=\hat w-G$ is of class
$\Ww^{k,q}$ on each
bounded open subset of $U$.
Since also $\p_t\hat w \in\Ww^{k,q}(\Omega)$
by step~VI), the function
$\p_tu=\p_t\hat w-\p_tG$ is of class
$\Ww^{k,q}$ on each
bounded open subset of $U$,
in particular on $\Omega^\prime$.
This concludes the proof of
theorem~\ref{thm:local-regularity}.
\end{proof}

\subsection*{Interior estimates}
\label{subsec:interior-estimates}

\begin{theorem}[Interior estimates]
\label{thm:interior-estimate}
Fix an integer $k\ge 0$
and constants $1<q<\infty$ and $0<r<R$.
Define
$\Omega_r=(-r^2,0]\times(-r,r)$.
Then there is a constant
$c=c(k,q,R-r)$ such that
\begin{equation}\label{eq:interior-estimate}
     \Norm{u}_{\Ww^{k+1,q}(\Omega_r)}
     \le c\left(
     \Norm{\p_su-\p_t\p_tu}_{\Ww^{k,q}(\Omega_R)}
     +\Norm{u}_{L^q(\Omega_R)}
     +\Norm{\p_tu}_{L^q(\Omega_R)}
     \right)
\end{equation}
for every $u\in C^\infty(\overline{\Omega_R})$.
\end{theorem}

\begin{proof}
The proof is by induction on $k$.

\noindent
{\it Case $k=0$.}
Fix a smooth compactly supported cutoff function
$\beta:\Omega_R\to[0,1]$
such that $\beta=1$ on $\Omega_r$.
Then
\begin{equation*}
\begin{split}
    &\Norm{u}_{\Ww^{1,q}(\Omega_r)}\\
    &\le \Norm{\beta u}_{L^q(\Omega_R)}
     +\Norm{\p_t\left(\beta u\right)}_{L^q(\Omega_R)}
     +\Norm{\p_t\p_t\left(\beta u\right)}_{L^q(\Omega_R)}
     +\Norm{\p_s\left(\beta u\right)}_{L^q(\Omega_R)}\\
    &\le 2R(1+2R)
     \Norm{\p_t\p_t\left(\beta u\right)}_{L^q(\Omega_R)}
     +\Norm{\p_s\left(\beta u\right)}_{L^q(\Omega_R)}\\
    &\le c \Norm{(\p_s-\p_t\p_t)\beta u}_{L^q(\Omega_R)}\\
    &\le c \Norm{(\p_s-\p_t\p_t)u}_{L^q(\Omega_R)}
     +C\left(\Norm{u}_{L^q(\Omega_R)}
     +\Norm{\p_tu}_{L^q(\Omega_R)}\right)
\end{split}
\end{equation*}
where
$c=c_q\left(1+2R(1+2R)\right)$ with
$c_q$ being the constant in 
theorem~\ref{thm:kerD-CalZyg}
and 
$$
     C=\Norm{\p_s\beta}_\infty
     +\Norm{\p_t\p_t\beta}_\infty
     +2\Norm{\p_t\beta}_\infty.
$$
The first step uses the fact
that $\beta=1$ on $\Omega_r$,
the definition
of the $\Ww^{1,q}$ norm, and monotonicity
of the integral.
To obtain step two we fixed $s$ and applied
Poincar\'{e}'s inequality
lemma~\ref{le:poincare-ineq}
to the functions $\beta u,
\p_t(\beta u)\in C_0^\infty(-R,R)$,
then we integrated over $s\in(-R^2,0]$.
Step three is by theorem~\ref{thm:kerD-CalZyg}.

\vspace{.05cm}
\noindent
{\it Induction step $k-1\Rightarrow k$.}
Fix $k\ge 1$.
It suffices to estimate the $\Ww^{k+1,q}$
norms of $u$, $\p_tu$, $\p_t\p_tu$, and $\p_su$ 
individually by the right hand side
of~(\ref{eq:interior-estimate}).
We provide details for the least trivial term
and leave the others as an exercise.
Fix constants $r<r_1<r_2<R$.
Then by the induction hypothesis in the case $k-1$
for the pair of sets $\Omega_r\subset\Omega_{r_1}$
and the function $v=\p_su$ we obtain that
\begin{equation*}
\begin{split}
    &\Norm{\p_su}_{\Ww^{k,q}(\Omega_r)}\\
    &\le c_1\left(
     \Norm{(\p_s-\p_t\p_t)\p_su}_{\Ww^{k-1,q}(\Omega_{r_1})}
     +\Norm{\p_su}_{L^q(\Omega_{r_1})}
     +\Norm{\p_t\p_su}_{L^q(\Omega_{r_1})}
     \right)\\
    &\le c_1\left(
     \Norm{(\p_su-\p_t\p_t)u}_{\Ww^{k,q}(\Omega_R)}
     +\Norm{u}_{\Ww^{1,q}(\Omega_{r_1})}
     +\Norm{\p_tu}_{\Ww^{1,q}(\Omega_{r_1})}
     \right)
\end{split}
\end{equation*}
for some constant $c_1=c_1(k-1,q,r_1-r)$.
To deal with the last term in the sum
we apply the case $k=0$
for the pair of sets $\Omega_{r_1}\subset\Omega_{r_2}$
and the function $v=\p_tu$ to obtain that
\begin{equation*}
\begin{split}
    &\Norm{\p_tu}_{\Ww^{1,q}(\Omega_{r_1})}\\
    &\le c_2\left(
     \Norm{(\p_s-\p_t\p_t)\p_tu}_{L^q(\Omega_{r_2})}
     +\Norm{\p_tu}_{L^q(\Omega_{r_2})}
     +\Norm{\p_t\p_tu}_{L^q(\Omega_{r_2})}
     \right)\\
    &\le c_2\left(
     \Norm{(\p_su-\p_t\p_t)u}_{\Ww^{k,q}(\Omega_R)}
     +\Norm{\p_tu}_{L^q(\Omega_R)}
     +\Norm{u}_{\Ww^{1,q}(\Omega_{r_2})}
     \right)
\end{split}
\end{equation*}
for some constant $c_2=c_2(q,r_2-r_1)$.
It remains to estimate the last term in the sum.
We apply again the case $k=0$, but now
for the pair of sets $\Omega_{r_2}\subset\Omega_R$
and the function $u$ to obtain that
$$
     \Norm{u}_{\Ww^{1,q}(\Omega_{r_2})}
     \le c_3\left(
     \Norm{(\p_s-\p_t\p_t)u}_{L^q(\Omega_R)}
     +\Norm{u}_{L^q(\Omega_R)}
     +\Norm{\p_tu}_{L^q(\Omega_R)}
     \right)
$$
for some constant $c_3=c_3(q,R-r_2)$.
This proves theorem~\ref{thm:interior-estimate}.
\end{proof}

\begin{proof}[Proof of 
theorem~\ref{thm:interior-regularity}]
a) Suppose the
parabolic rectangle
$\Omega=(\sigma-r^2,\sigma]\times(\tau-r,\tau+r)$
is contained in
the cylinder $Z_T=(-T,0]\times S^1$.
Then the assumptions
of theorem~\ref{thm:local-regularity}~a)
are satisfied for
the restrictions of $u$ and $f$
to $\Omega$ and therefore
$u\in\Ww^{k+1,q}_{loc}(\Omega)$.
Now every compact subset of $Z_T$
can be covered by finitely many
parabolic rectangles.
Hence $u$ is locally $\Ww^{k+1,q}$
integrable on $Z_T$.

b) Induction over $k$
based on theorem~\ref{thm:interior-estimate}
and a covering argument by parabolic rectangles
proves~b).
\end{proof}

\section{Parabolic bootstrapping}
\label{sec:bootstrapping}

In this section
we establish
uniform Sobolev bounds for strong
solutions $u$ of the heat
equation~(\ref{eq:heat-local-F})
by parabolic bootstrapping.
This immediately implies
theorem~\ref{thm:bootstrap}.
In order to deal with
the heat equation's
quadratic nonlinearity
in $\p_tu$
we first prove in lemma~\ref{le:bootstrap}
apriori continuity of $\p_tu$.
Then the heat equation
can be treated like a linear equation
in the crucial first step $\ell=1$
of the parabolic bootstrap.

In this section we fix
a closed smooth submanifold
$M\hookrightarrow\R^N$
and a smooth family of
vector-valued symmetric bilinear forms
$\Gamma:M\to\R^{N\times N\times N}$.
Recall that the cylinders $Z=Z_T$
and $Z^\prime=Z_{T^\prime}$
are defined by~(\ref{eq:Z}).

\begin{lemma}[Apriori continuity of $\p_tu$]
\label{le:bootstrap}
Fix constants
$p>2$, $\mu_0>1$, and $T>0$.
Fix a map $F:Z\to\R^N$ such that
$F$ and $\p_tF$ are of class $L^p$.
Assume that $u:Z\to\R^N$ is a $\Ww^{1,p}$
map taking values in $M$
with $\norm{u}_{\Ww^{1,p}}\le\mu_0$ and
such that the perturbed heat equation
\begin{equation}\label{eq:heat-local-F}
     \p_su-\p_t\p_tu
     =\Gamma(u)\left(\p_tu,\p_tu\right)
     +F
\end{equation}
is satisfied almost everywhere.
Then $\p_tu$ is continuous.
More precisely, for every
$T^\prime\in(0,T)$
there is a constant
$c=c(p,\mu_0,T,T^\prime,\norm{\Gamma}_{C^1})$
such that
$$
     \Norm{\p_tu}_{C^0(Z^\prime)}
     \le c\left(1+\Norm{\p_tF}_{L^p(Z)}\right).
$$
\end{lemma}

Note that by the Sobolev embedding
theorem the assumption $p>2$ guarantees
that the $\Ww^{1,p}$ map $u$
is continuous. Hence it makes sense
to specify that $u$ takes values in
the submanifold $M$ of $\R^N$.
Abbreviate
$\Ww^{k,p}(Z)=\Ww^{k,p}(Z,\R^N)$.

\begin{remark}\label{rmk:unlikely}
Since the proof of
lemma~\ref{le:bootstrap}
relies heavily on
the product estimate
theorem~\ref{thm:product-estimate}
it seems unlikely that
the assumption $u\in\Ww^{1,p}$
can be weakened to $u\in W^{1,p}$ --
unless we also replace
the assumption $p>2$ by $p>3$.
\end{remark}

\begin{proposition}\label{prop:bootstrap}
Under the assumptions of
lemma~\ref{le:bootstrap}
the following is true
for every integer $k\ge1$ such that
$F$ and $\p_tF$ are in $\Ww^{k-1,p}(Z)$
and every constant 
$T^\prime\in(0,T)$.
\begin{enumerate}

\item[\rm(i)]
There is a constant
$a_k$ depending on $p$, $\mu_0$,
$T$, $T^\prime$, $\norm{\Gamma}_{C^{2k+2}}$, and
the $\Ww^{k-1,p}(Z)$ norms of
$F$ and $\p_tF$ such that
$$
     \Norm{\p_tu}_{\Ww^{k,p}(Z^\prime)}
     \le a_k.
$$

\item[\rm(ii)]
If $\p_sF\in\Ww^{k-1,p}(Z)$
then there is a constant
$b_k$ depending on $p$, $\mu_0$,
$T$, $T^\prime$,
$\norm{\Gamma}_{C^{2k+2}}$,
and the $\Ww^{k-1,p}(Z)$ norms
of $F$, $\p_tF$, and
$\p_sF$ such that
\begin{equation*}
     \Norm{\p_su}_{\Ww^{k,p}(Z^\prime)}
     \le b_k.
\end{equation*}

\item[\rm(iii)]
If $\p_t\p_tF\in\Ww^{k-1,p}(Z)$
then there is a constant
$c_k$ depending on $p$, $\mu_0$,
$T$, $T^\prime$,
$\norm{\Gamma}_{C^{2k+2}}$, and
the $\Ww^{k-1,p}(Z)$ norms
of $F$, $\p_tF$, and
$\p_t\p_tF$ such that
\begin{equation*}
     \Norm{\p_t\p_tu}_{\Ww^{k,p}(Z^\prime)}
     \le c_k.
\end{equation*}

\end{enumerate}
\end{proposition}

\begin{notation}
In the proofs of lemma~\ref{le:bootstrap}
and proposition~\ref{prop:bootstrap}
we use the following notation.
The parabolic $\Cc^k$ norm of a smooth
function $u$ is given by
\begin{equation}\label{eq:parabolic-Ck}
     \Norm{u}_{\Cc^k}
     :=\sum_{2\nu+\mu\le 2k}
     \Norm{\p_s^\nu \p_t^\mu u}_{\infty}.
\end{equation}
Compare this to standard space $C^k$ with norm
$\norm{u}_{C^k}=\sum_{\nu+\mu\le k}
\norm{\p_s^\nu \p_t^\mu u}_{\infty}$.
Given two constants $T>T^\prime>0$ consider
the sequence
\begin{equation}\label{eq:T_k}
     T_k:=T^\prime+\frac{T-T^\prime}{k},\quad
     k\in\N.
\end{equation}
Note that $T_1=T$.
This definition
also makes sense if
we replace $k$ by
a real number $r\ge1$.
Now consider the cylinders
$Z_r=(-T_r,0]\times S^1$.
By $\INT\, Z_r$ we denote the interior
$(-T_r,0)\times S^1$ of $Z_r$.
It is useful to
memorize that $Z_{r+1}\subset Z_r$.
For each positive integer $k$
fix a smooth compactly supported
cutoff function
\begin{equation}\label{eq:cutogg-rho}
     \rho_k:(-T_k,0]\to[0,1]
\end{equation}
such that $\rho_k=1$ on
$Z_{k+1}$ and $\norm{\p_s\rho}_\infty\ge 1$.
\end{notation}

\begin{proof}
[Proof of lemma~\ref{le:bootstrap}]
Denote the nonlinear part of the heat
equation~(\ref{eq:heat-local-F}) by
$$
    h=h(u)
    =\Gamma(u)\left(\p_tu,\p_tu\right)+F
$$
and the first cutoff function
fixed in~(\ref{eq:cutogg-rho})
by $\rho=\rho_1$.
Then $h\in L^p(Z_2)$, namely
\begin{equation*}
\begin{split}
     \Norm{h}_{L^p(Z_2)}
    &\le\Norm{\rho^2h}_{L^p(Z_1)}\\
    &\le \Norm{\Gamma}_\infty
     \Norm{\Abs{\p_t(\rho u)}\cdot
     \Abs{\p_t(\rho u)}}_{L^p(Z_1)}
     +\Norm{\rho^2F}_{L^p(Z_1)}\\
    &\le C_p\Norm{\Gamma}_\infty\Norm{\p_s\rho}_\infty^2
     \Norm{u}_{\Ww^{1,p}(Z)}^2
     +\Norm{F}_{L^p(Z)}
\end{split}
\end{equation*}
where in step one and two
we used that $\rho^2=1$ on $Z_2$
and independence of $\rho$
on the $t$ variable, respectively.
The last step is by
the product estimate
theorem~\ref{thm:product-estimate}
with constant $C_p>0$
applied to the compactly supported
$\Ww^{1,p}$ map $\rho u:Z\to\R^N$
using density.
Compactness of $M$ implies
that $\norm{\Gamma}_\infty<\infty$.
Next observe that
\begin{equation}\label{eq:p_t-h}
     \p_th
     =d\Gamma(u) \left(\p_tu,\p_tu,\p_tu\right)
     +2\Gamma(u) \left(\p_t\p_tu,\p_tu\right)
     +\p_tF.
\end{equation}

Now we indicate the main idea of proof.
Suppose we knew that
$\p_th\in L^\chi(Z_{k+1})$
for some $\chi>1$ and some $k\in\N$, then
\begin{equation}\label{eq:p_tu}
\begin{split}
     \int_{Z_{k+1}}
     \p_tu\left(-\p_s\phi-\p_t\p_t\phi\right)
    &=-\int_{Z_{k+1}} \p_su\,\p_t\phi
     +\int_{Z_{k+1}} \p_t\p_tu\,\p_t\phi \\
    &=-\int_{Z_{k+1}} h\,\p_t\phi\\
    &=\int_{Z_{k+1}} \p_th\,\phi
\end{split}
\end{equation}
for every
$\phi\in C^\infty_0(\INT\, {Z_{k+1}})$.
Here all steps use
integration by parts.
Step two is by
definition of $h$
and the assumption that
$u$ satisfies
the heat equation~(\ref{eq:heat-local-F})
almost everywhere.
Theorem~\ref{thm:interior-regularity}
on interior regularity
then asserts that
$\p_tu\in\Ww^{1,\chi}(Z_{k+2})$.
Hence we have improved
the regularity of $\p_tu$
which in turn
improves the of regularity $\p_th$
as given by~(\ref{eq:p_t-h}).
Now start over again.
We prove below that under this
iteration $\chi$ eventually
converges to $p$.
But $p>2$, hence continuity
of $\p_tu$ follows by the
Sobolev embedding
$\W^{1,\chi}\hookrightarrow C^0$.

To get the iteration started at $k=1$
we need to first prove
that $\p_th\in L^\chi(Z_2)$
for some $\chi>1$.
As a first try recall that $u\in\Ww^{1,p}(Z_1)$
by assumption,
therefore the first term in~(\ref{eq:p_t-h})
is in $L^{p/3}$ only whereas the second term
is in $L^{p/2}$.
Hence $\p_th\in L^{p/3}$,
but $p/3$ is not necessarily
larger than $1$.
Fortunately,
using the product estimate
theorem~\ref{thm:product-estimate}
we can do better.
Recall that $p>2$ is fixed by assumption.
Consider the function
$$
     \chi=\chi_p(q)=\frac{pq}{p+q}
$$
and observe that $1/p+1/q=1/\chi$.
Apply H\"older's inequality
to obtain
\begin{equation}\label{eq:p_th}
\begin{split}
     \Norm{\p_th}_{L^\chi(Z_{k+1})}
    &\le \Norm{{\rho_k}^2\p_th}_{L^\chi(Z_k)}
    \\
    &\le \Norm{d\Gamma}_\infty
      \Norm{\Abs{\p_t(\rho_ku)}\cdot\Abs{\p_t(\rho_ku)}}
      _{L^p(Z_k)}
      \Norm{\p_tu}_{L^q(Z_k)}
    \\
    &\quad 
      +2\Norm{\Gamma}_\infty
     \Norm{\p_t\p_tu}_{L^p(Z_k)}
     \Norm{\p_tu}_{L^q(Z_k)}
     +\Norm{\p_tF}_{L^\chi(Z_k)}
    \\
    &\le C_p\Norm{d\Gamma}_\infty
     \Norm{\p_s\rho_k}_\infty^2
     \Norm{u}_{\Ww^{1,p}(Z)}^2
     \Norm{\p_tu}_{L^q(Z_k)}
    \\
    &\quad
     +2\Norm{\Gamma}_\infty
     \Norm{\p_t\p_tu}_{L^p(Z)}
     \Norm{\p_tu}_{L^q(Z_k)}
     +\Norm{\p_tF}_{L^p(Z_k)}
    \\
    &\le\alpha\Norm{\p_tu}_{L^q(Z_k)}
     +\Norm{\p_tF}_{L^p(Z)}.
\end{split}
\end{equation}
Here the third step is by
the product estimate
theorem~\ref{thm:product-estimate}
with constant $C_p$
and the constant $\alpha$
in the last line
depends on $p$, $\mu_0$,
$\norm{\Gamma}_{C^1}$, and $\rho_k$.
We used again one of the cutoff functions
in~(\ref{eq:cutogg-rho}) to produce
a compactly supported function
as required by
the product estimate.
Consequently the domain shrinks.

Now we start the iteration
with initial value $q_1=p$.
Then $\chi(q_1)=p/2>1$.
Hence $\p_th\in L^{p/2}(Z_2)$
by~(\ref{eq:p_th}) for $k=1$.
Therefore by~(\ref{eq:p_tu})
theorem~\ref{thm:interior-regularity}
applies for the functions $\p_tu$ and $f=\p_th$
and proves that
$\p_tu\in\Ww^{1,p/2}_{loc}(Z_2)$
and
\begin{equation}\label{eq:p_tu2}
\begin{split}
     \Norm{\p_tu}_{\Ww^{1,p/2}(Z_{3})}
    &\le c_2\left(
     \Norm{\p_th}_{L^{p/2}(Z_2)}
     +\mu_0\right)\\
    &\le c_2\left( \alpha\mu_0
     +\Norm{\p_tF}_{L^p(Z)}
     +\mu_0\right)
\end{split}
\end{equation}
for some constant
$c_2=c_2(p,T_2-T_3)$.
Step two uses~(\ref{eq:p_th})
for $k=1$ and $q=p/2$,
the fact that
$\norm{\p_tu}_{p/2}
\le\norm{\p_tu}_p$,
and the assumption
$\norm{\p_tu}_p\le\mu_0$.

\noindent
Now there are three cases:
If $p>4$ then we are done by
the Sobolev embedding
$W^{1,p/2}\hookrightarrow C^0$
on the domain $Z_3$;
see e.g.~\cite[App.~B.1]{MS}
for the relevant embedding theorems.
If $p<4$, then
the value of
$\chi=\chi_p(q_1)=p/2$ is in the interval
$(1,2)$ and in this case there is the 
Sobolev embedding
$$
     \Ww^{1,\chi}(Z_{3})\subset W^{1,\chi}(Z_{3})
     \hookrightarrow L^{2\chi/(2-\chi)}(Z_{3})
     =L^{q_2}(Z_3)
$$
with constant $C_2=C_2(p,T_{3})>0$.
Here we abbreviated
$$
     q_2:=\frac{2\chi}{2-\chi}
     =\frac{2pq_1}{2p+2q_1-pq_1}
     =\frac{2p}{4-p}.
$$
Hence $\p_tu\in L^{2p/(4-p)}(Z_3)$.
Since $2p/(4-p)>p$ is equivalent to $2<p<4$,
this means that
the regularity of $\p_tu$
has been improved -- on the
expense of a smaller domain though.
The case $p=4$ means that $u:Z\to\R^N$
is a $\Ww^{1,4}$ map to start with.
But then it is also a $\Ww^{1,3}$
map and we are in the former case.

Repeating the same argument
with new initial value $q_2$
proves that
$\p_tu\in\Ww^{1,\chi_p(q_2)}(Z_5)$.
Again this space embedds either
in $C^0(Z_5)$ and we are done
or it embedds in $L^{q_3}(Z_5)$ where
$q_3=2pq_2/(2p+2q_2-pq_2)>q_2$.
It is crucial that
in~(\ref{eq:p_th})
the value of $p$
is fixed. Firstly, because
the product estimate 
theorem~\ref{thm:product-estimate}
requires $p\ge2$
and, secondly, because 
we only know that $\p_t\p_tu\in L^p$.
Proceeding this way we obtain the
sequence $q_k$ determined by
\begin{equation}\label{eq:q_k}
     q_{k+1}=\frac{2pq_k}{2p+2q_k-pq_k},\qquad
     q_1=p.
\end{equation}
Observe again that the condition
$p>2$ implies that
$q_{k+1}>q_k$.
Hence the sequence is strictly
monotone increasing.
Next we prove that
$q_k\to\infty$ as $k\to\infty$.
Assume by contradiction
that this is not true.
Then by strict monotonicity
the sequence is bounded
and admits a unique limit, say $q$.
By~(\ref{eq:q_k})
this limit satisfies
$q=2pq/(2p+2q-pq)$.
But this is equivalent to $p=2$
contradicting $p>2$.
It follows that $\chi_p(q_k)$ converges to $p$
as $k\to\infty$.
But $p>2$, hence
whenever $k$ is sufficiently large
there is the Sobolev embedding
$$
     \Ww^{1,\chi_p(q_k)}(Z_{2k+1})
     \hookrightarrow C^0(Z_{2k+1})
     \subset C^0(Z^\prime)
$$
and this implies the estimate
in lemma~\ref{le:bootstrap}.
Clearly $\p_tu$ is continuous on
the whole cylinder $Z$
since every point is contained
in some subcylinder $Z^\prime$.
\end{proof}

\begin{proof}
[Proof of proposition~\ref{prop:bootstrap}]
We prove the following claim
by induction on $\ell$.
Recall from~(\ref{eq:T_k})
the definition of the reals $T_\ell$ 
and the cylinders $Z_\ell$.

\vspace{.1cm}
\noindent
{\it Claim.}
{\it Given $0<T^\prime<T$ and
$k\ge1$ such that
$F$ and $\p_tF$ are in
$\Ww^{k-1,p}$,
then the following is true
for every
$\ell\in\{1,\ldots,k\}$.
\begin{enumerate}

\item[\rm(a)]
$\p_tu\in\Ww^{\ell,p}_{loc}(Z_{3\ell-1})$
and there exists a constant
$A_\ell$ depending on
$p$, $\mu_0$,
$\norm{\Gamma}_{C^{2\ell+2}}$,
$\norm{F}_{\Ww^{\ell-1,p}}$,
and $\norm{\p_tF}_{\Ww^{\ell-1,p}}$
such that
$$
     \Norm{\p_tu}
     _{\Ww^{\ell,p}(Z_{3\ell})}
     \le A_\ell.
$$

\item[\rm(b)]
If $\p_sF\in\Ww^{k-1,p}$
then $\p_su\in\Ww^{\ell,p}_{loc}(Z_{3\ell})$
and there exists a constant
$B_\ell$ depending on
$p$, $\mu_0$,
$\norm{\Gamma}_{C^{2\ell+2}}$,
$\norm{F}_{\Ww^{\ell-1,p}}$,
$\norm{\p_tF}_{\Ww^{\ell-1,p}}$,
and $\norm{\p_sF}_{\Ww^{\ell-1,p}}$
such that
$$
     \Norm{\p_su}
     _{\Ww^{\ell,p}(Z_{3\ell+1})}
     \le B_\ell.
$$

\item[\rm(c)]
If $\p_t\p_tF\in\Ww^{k-1,p}$ then
$\p_t\p_tu\in\Ww^{\ell,p}_{loc}(Z_{3\ell+1})$
and there exists a constant
$C_\ell$ depending on
$p$, $\mu_0$,
$\norm{\Gamma}_{C^{2\ell+2}}$,
$\norm{F}_{\Ww^{\ell-1,p}}$,
$\norm{\p_tF}_{\Ww^{\ell-1,p}}$,
and $\norm{\p_t\p_tF}_{\Ww^{\ell-1,p}}$
such that
$$
    \Norm{\p_t\p_tu}_{\Ww^{\ell,p}(Z_{3\ell+2})}
    \le C_\ell.
$$
\end{enumerate}
}

Here and throughout the domain
of all norms is $Z=Z_T$,
unless specified otherwise.
An exception are the various norms
of $\Gamma$ appearing below
for which the domain
is the compact manifold $M$.
We abbreviate
$h=\Gamma(u)\left(\p_tu,\p_tu\right)+F$.

\vspace{.1cm}
\noindent
{\it Case $\ell=1$.}
%
By lemma~\ref{le:bootstrap}
with $T^\prime=T_2$
there is a constant
$C_0$ depending on $p$, $\mu_0$, $T$, $T_2$,
and $\norm{\Gamma}_{C^1}$,
such that
\begin{equation}\label{eq:C0}
     \Norm{\p_tu}_{C^0(Z_2)}
     \le C_0\left(1+\Norm{\p_tF}_p\right).
\end{equation}

\vspace{.1cm}
\noindent
(a) Recall that $\p_th$ is given 
by~(\ref{eq:p_t-h}).
Straightforward
calculation shows that
\begin{equation*}
\begin{split}
     \Norm{\p_th}_{L^p(Z_2)}
    &\le
     \Norm{d\Gamma}_\infty
     \Norm{\p_tu}^2_{C^0(Z_2)}
     \Norm{\p_tu}_{L^p(Z_2)}
     +\Norm{\p_tF}_{L^p(Z_2)}\\
    &\quad
     +2\Norm{\Gamma}_\infty
     \Norm{\p_tu}_{C^0(Z_2)}
     \Norm{\p_t\p_tu}_{L^p(Z_2)}\\
    &\le \alpha\left(1+
     \Norm{\p_tF}_p^2\right)
\end{split}
\end{equation*}
for some constant
$\alpha=\alpha(p,\mu_0,T,T_2,
\norm{\Gamma}_{C^1})$.
We used~(\ref{eq:C0})
and the assumption
$\norm{u}_{\Ww^{1,p}}\le\mu_0$.
Recall from~(\ref{eq:p_tu})
that $\p_tu$ satisfies
$$
     \int_{Z_2} \p_tu\left(-\p_s\phi-\p_t\p_t\phi\right)
     =\int_{Z_2} \p_th\, \phi
$$
for every
$\phi\in C^\infty_0(\INT\, Z_2)$.
Hence theorem~\ref{thm:interior-regularity}
on interior regularity for $q=p$,
$T=T_2$, $T^\prime=T_3$, $k=0$,
and the functions $f=\p_th$ and $\p_tu$
in $L^p(Z_2)$
proves that
$\p_tu\in\Ww^{1,p}_{loc}(Z_2)$
and
$$
     \Norm{\p_tu}_{\Ww^{1,p}(Z_3)}
     \le \mu\left(
     \Norm{\p_th}_{L^p(Z_2)}
     +\Norm{\p_tu}_{L^p(Z_2)}\right)
$$
for some constant
$\mu=\mu(p,T_2,T_3)$.
Now use the estimate for $\p_th$ to see that
$$
     \Norm{\p_tu}_{\Ww^{1,p}(Z_3)}
     \le A\left(1+\Norm{\p_tF}_p^2\right)
$$
for some constant
$A=A(p,\mu_0,T,T_2,T_3,\norm{\Gamma}_{C^1})$.

\vspace{.1cm}
\noindent
(b) Straightforward
calculation shows that
\begin{equation*}
\begin{split}
     \Norm{\p_sh}_{L^p(Z_3)}
    &\le
     \Norm{d\Gamma}_\infty
     \Norm{\p_tu}^2_{C^0(Z_3)}
     \Norm{\p_su}_{L^p(Z_3)}
     +\Norm{\p_sF}_{L^p(Z_3)}\\
    &\quad
     +2\Norm{\Gamma}_\infty
     \Norm{\p_tu}_{C^0(Z_3)}
     \Norm{\p_s\p_tu}_{L^p(Z_3)}\\
    &\le \beta\left(1
     +\Norm{\p_tF}_p^3
     \right)
     +\Norm{\p_sF}_p
\end{split}
\end{equation*}
for some constant
$\beta=\beta(p,\mu_0,T,T_2,T_3,
\norm{\Gamma}_{C^1})>1$.
Here we estimated the $L^p$ norm
of $\p_s\p_tu$ by the
$\Ww^{1,p}$ estimate for $\p_tu$
just proved in~(a).
We also used the $C^0$ estimate~(\ref{eq:C0}).
Next observe that
\begin{equation}\label{eq:p_su}
\begin{split}
     \int_{Z_3} \p_su
     \left(-\p_s\phi-\p_t\p_t\phi\right)
    &=-\int_{Z_3}
     \left(\p_su-\p_t\p_tu\right)
     \p_s\phi\\
    &=-\int_{Z_3} \left(\Gamma(u)
     \left(\p_tu,\p_tu\right)+F(u)\right)
     \p_s\phi\\
    &=\int_{Z_3} \p_sh\,\phi
\end{split}
\end{equation}
for every
$\phi\in C^\infty_0(\INT\, Z_3)$.
Here steps one and three
are by integration by parts.
Step two uses
the assumption that $u$
satisfies
the heat equation~(\ref{eq:heat-local-F})
almost everywhere.
Now theorem~\ref{thm:interior-regularity}
proves that
$\p_su\in\Ww^{1,p}_{loc}(Z_3)$
and
$$
     \Norm{\p_su}_{\Ww^{1,p}(Z_4)}
     \le \mu\left(
     \Norm{\p_sh}_{L^p(Z_3)}
     +\Norm{\p_su}_{L^p(Z_3)}\right)
$$
for some constant
$\mu=\mu(p,T_3,T_4)$.
Now use the estimate for $\p_sh$ to see that
$$
     \Norm{\p_su}_{\Ww^{1,p}(Z_4)}
     \le B\left(1+\Norm{\p_tF}_p^3
     +\Norm{\p_sF}_p\right)
$$
for some constant
$B=B(p,\mu_0,T,T_2,T_3,T_4,\norm{\Gamma}_{C^1})$.

\vspace{.1cm}
\noindent
(c) Straighforward calculation
shows that
\begin{equation*}
\begin{split}
     \Norm{\p_t\p_th}_{L^p(Z_4)}
    &\le
     \Norm{d^2\Gamma}_\infty
     \Norm{\p_tu}^3_{C^0(Z_4)}
     \Norm{\p_tu}_{L^p(Z_4)}
     +\Norm{\p_t\p_tF}_{L^p(Z_4)}
     \\
    &\quad
     +4\Norm{d\Gamma}_\infty
     \Norm{\p_tu}^2_{C^0(Z_4)}
     \Norm{\p_t\p_tu}_{L^p(Z_4)}\\
    &\quad
     +2\Norm{\Gamma}_\infty
     \Norm{\p_tu}_{C^0(Z_4)}
     \Norm{\p_t\p_t\p_tu}_{L^p(Z_4)}
     \\
    &\quad
     +2\Norm{\Gamma}_\infty
     \Norm{\p_t\p_tu}_{C^0(Z_4)}
     \Norm{\p_t\p_tu}_{Lp(Z_4)}
     \\
    &\le \gamma\left(1
     +\Norm{\p_tF}_p^4
     \right)
     +\Norm{\p_t\p_tF}_p
\end{split}
\end{equation*}
for some constant
$\gamma=\gamma(p,\mu_0,T,T_2,T_3,T_4,
\norm{\Gamma}_{C^2})$.
In the final inequality
we used the $C^0$ estimate~(\ref{eq:C0})
for $\p_tu$
and the $\Ww^{1,p}$
estimate for $\p_tu$
proved above in~(a).
This takes care of all terms but one,
namely the $C^0$ norm of $\p_t\p_tu$.
Here we use that
$\p_t\p_t\p_tu$ and
$\p_s\p_t\p_tu=\p_t\p_t\p_su$
are in $L^p(Z_4)$
by~(a) and~(b), respectively.
Hence $\p_t\p_tu\in C^0$
by the Sobolev embedding
$W^{1,p}\hookrightarrow C^0$.
Similarly to the calculation
in~(\ref{eq:p_tu})
it follows that
$$
     \int_{Z_4}\p_t\p_tu
     \left(-\p_s\phi-\p_t\p_t\phi\right)
     =\int_{Z_4}\p_t\p_th\,\phi
$$
for every $\phi\in C_0^\infty(\INT\, Z_4)$.
Theorem~\ref{thm:interior-regularity}
then proves that
$\p_t\p_tu\in\Ww^{1,p}_{loc}(Z_4)$
and
$$
     \Norm{\p_t\p_tu}_{\Ww^{1,p}(Z_5)}
     \le \mu\left(
     \Norm{\p_t\p_th}_{L^p(Z_4)}
     +\Norm{\p_t\p_tu}_{L^p(Z_4)}\right)
$$
for some constant
$\mu=\mu(p,T_4,T_5)$.
Now use the estimate for $\p_t\p_th$ to see that
$$
     \Norm{\p_t\p_tu}_{\Ww^{1,p}(Z_5)}
     \le C\left(1+\Norm{\p_tF}_p^4
     +\Norm{\p_t\p_tF}_p\right)
$$
for some constant
$C=C(p,\mu_0,T,T_2,T_3,T_4,T_5,\norm{\Gamma}_{C^2})$.

\vspace{.1cm}
\noindent
{\it Induction step
$\ell\Rightarrow \ell+1$.}
%
Fix an integer $\ell\in\{1,\ldots,k-1\}$
and assume that~(a--c)
are true for this choice of $\ell$.
We indicate this by the
notation~$\text{(a--c)}_\ell$.
The task at hand is
to prove~$\text{(a--c)}_{\ell+1}$.
Recall the parabolic $\Cc^\ell$
norm~(\ref{eq:parabolic-Ck}).
An immediate consequence
of the induction
hypothesis~$\text{(a--c)}_\ell$
is that
\begin{equation*}
     \Norm{u}_{\Ww^{\ell+1,p}(Z_{3\ell+2})}
     \le D_{\ell+1}^\prime
\end{equation*}
for some constant
$D_{\ell+1}^\prime=D_{\ell+1}^\prime(p,\mu_0,
\norm{\Gamma}_{C^{2\ell+2}},
\norm{F}_{\Ww^{\ell,p}})$.
Hence
\begin{equation}\label{eq:Cell}
     \Norm{u}_{\Cc^\ell(Z_{3\ell+2})}
     \le D_{\ell+1}
\end{equation}
for some constant
$D_{\ell+1}=D_{\ell+1}(p,\mu_0,
\norm{\Gamma}_{C^{2\ell+2}},
\norm{F}_{\Ww^{\ell,p}})$.
To see this observe that
up to a constant
the $\Cc^\ell$ norm
can be estimated
by the $\Ww^{\ell+1,p}$ norm.
(This boils down to the
Sobolev embedding
$W^{1,p}\hookrightarrow C^0$ 
for each individual derivative of $u$
showing up in $\Cc^\ell$.)

\vspace{.1cm}
\noindent
$\text{(a)}_{\ell+1}$ Straightforward
calculation shows that
\begin{equation*}
\begin{split}
    &\Norm{\p_th}_{\Ww^{\ell,p}(Z_{3\ell+2})}\\
    &\le
     \Norm{d\Gamma}_{C^{2\ell}} d_\ell
     \Norm{u}_{\Cc^\ell(Z_{3\ell+2})}^2
     \Norm{\p_tu}_{\Ww^{\ell,p}(Z_{3\ell+2})}
     +\Norm{\p_tF}
     _{\Ww^{\ell,p}(Z_{3\ell+2})}
     \\
    &\quad
     +2\Norm{\Gamma}_{C^{2\ell}} d_\ell
     \Norm{u}_{\Cc^\ell(Z_{3\ell+2})}
     \left(
     \Norm{\p_tu}
     _{\Ww^{\ell,p}(Z_{3\ell+2})}
     +\Norm{\p_t\p_tu}
     _{\Ww^{\ell,p}(Z_{3\ell+2})}
     \right)\\
    &\le \alpha_{\ell+1}
     +\Norm{\p_tF}_{\Ww^{\ell,p}}
\end{split}
\end{equation*}
for some constant
$\alpha_{\ell+1}=\alpha_{\ell+1}(p,\mu_0,
\norm{\Gamma}_{C^{2\ell+2}},
\norm{F}_{\Ww^{\ell,p}})$.
The first inequality
follows from the identity~(\ref{eq:p_t-h})
and the last two estimates of
corollary~\ref{cor:product-Sobolev}
with constant $d_\ell$.
Notice the difference
between the standard $C^\ell$ and
the parabolic $\Cc^\ell$ norms.
To obtain the second inequality
we applied~(\ref{eq:Cell})
and the induction
hypotheses~${\rm(a)}_\ell$
and~${\rm(c)}_\ell$
to estimate the
$\Ww^{\ell,p}$ norms of
$\p_tu$ and $\p_t\p_tu$,
respectively.
Next observe that
theorem~\ref{thm:interior-regularity}
applies by~(\ref{eq:p_tu})
and shows that
$\p_tu\in\Ww^{\ell+1,p}_{loc}(Z_{3\ell+2})$
and
$$
     \Norm{\p_tu}_{\Ww^{\ell+1,p}(Z_{3\ell+3})}
     \le \mu\left(
     \Norm{\p_th}_{\Ww^{\ell,p}(Z_{3\ell+2})}
     +\Norm{\p_tu}_{L^p(Z_{3\ell+2})}
     \right)
$$
for some constant
$\mu=\mu(p,Z_{3\ell+2},Z_{3\ell+3})$.
Now the assumption
$\norm{u}_{\Ww^{1,p}}\le\mu_0$
and the estimate for $\p_th$
conclude the proof
of~${\text{(a)}}_{\ell+1}$.
For latter reference we remark
that~${\text{(a)}}_{\ell+1}$
implies -- similarly to~(\ref{eq:Cell}) --
the estimate
\begin{equation}\label{eq:C^0-p_tu}
     \Norm{\p_tu}_{\Cc^\ell(Z_{3\ell+3})}
     \le E_\ell
\end{equation}
for some constant
$E_\ell=E_\ell
(p,\mu_0,\norm{\Gamma}_{C^{2\ell+2}},
\norm{F}_{\Ww^{\ell,p}},
\norm{\p_tF}_{\Ww^{\ell,p}})$.

\vspace{.1cm}
\noindent
$\text{(b)}_{\ell+1}$ Straightforward 
calculation using 
the $\Ww^{\ell+1,p}$ estimate
for $\p_tu$ just proved
and the induction 
hypotheses~${\text{(a--c)}}_\ell$
implies that
\begin{equation*}
\begin{split}
     \Norm{\p_sh}_{\Ww^{\ell,p}(Z_{3\ell+3})}
    &\le
     \Norm{d\Gamma}_{C^{2\ell}}
     \Norm{\p_tu}_{\Cc^\ell(Z_{3\ell+3})}^2
     \Norm{\p_su}_{\Ww^{\ell,p}(Z_{3\ell+3})}
     +\Norm{\p_sF}_{\Ww^{\ell,p}(Z_{3\ell+3})}
     \\
    &\quad
     +2\Norm{\Gamma}_{C^{2\ell}}
     \Norm{\p_tu}_{\Cc^\ell(Z_{3\ell+3})}
     \Norm{\p_s\p_tu}_{\Ww^{\ell,p}(Z_{3\ell+3})}
     \\
    &\le \beta_{\ell+1}
     +\Norm{\p_sF}_{\Ww^{\ell,p}}
\end{split}
\end{equation*}
for some constant
$\beta_{\ell+1}=\beta_{\ell+1}(p,\mu_0,
\norm{\Gamma}_{C^{2\ell+2}},
\Norm{F}_{\Ww^{\ell,p}},
\Norm{\p_tF}_{\Ww^{\ell,p}})$.
To obtain the first inequality
we simply pulled
out the $\Cc^\ell$ norms.
In the second inequality
we used~(\ref{eq:C^0-p_tu}),
the induction
hypothesis~${\text{(b)}}_\ell$
to estimate the $\Ww^{\ell,p}$ norm of $\p_su$,
and the induction
hypothesis~${\text{(a)}}_{\ell+1}$
just proved
to estimate the
$\Ww^{\ell,p}$ norm
of $\p_s\p_tu$.
Next observe that
theorem~\ref{thm:interior-regularity}
applies by the identity~(\ref{eq:p_su})
with $Z_3$ replaced by $Z_{3\ell+3}$
and shows that
$\p_su\in\Ww^{\ell+1,p}_{loc}(Z_{3\ell+3})$
and
$$
     \Norm{\p_su}_{\Ww^{\ell+1,p}(Z_{3\ell+4})}
     \le \mu\left(
     \Norm{\p_sh}_{\Ww^{\ell,p}(Z_{3\ell+4})}
     +\Norm{\p_tu}_{L^p(Z_{3\ell+4})}
     \right)
$$
for some constant
$\mu=\mu(p,Z_{3\ell+3},Z_{3\ell+4})$.
Now use the estimate for $\p_sh$.

\vspace{.1cm}
\noindent
$\text{(c)}_{\ell+1}$ Straighforward calculation
shows that
\begin{equation*}
\begin{split}
     \Norm{\p_t\p_th}
     _{\Ww^{\ell,p}(Z_{3\ell+3})}
    &\le
     \Norm{d^2\Gamma}_{C^{2\ell}}
     \Norm{\p_tu}^3_{\Cc^\ell}
     \Norm{\p_tu}_{\Ww^{\ell,p}}\\
    &\quad
      +5\Norm{d\Gamma}_{C^{2\ell}}
     \Norm{\p_tu}^2_{\Cc^\ell}
     \Norm{\p_t\p_tu}_{\Ww^{\ell,p}}\\
    &\quad
     +2\Norm{\Gamma}_{C^{2\ell}}
     \Norm{\p_tu}_{\Cc^\ell}
     \Norm{\p_t\p_t\p_tu}_{\Ww^{\ell,p}}
     +\Norm{\p_t\p_tF}_{\Ww^{\ell,p}}\\
    &\quad
     +2\Norm{\Gamma}_{C^{2\ell}} C_k^\prime
    \Norm{\p_tu}_{\Cc^\ell}
     \Norm{\p_t\p_tu}_{\Ww^{\ell,p}}.
\end{split}
\end{equation*}
Here all norms are taken on
the domain $Z_{3\ell+3}$
except those involving $\Gamma$
which are taken over $M$.
Notice that in the first three
terms of the sum we simply pulled out
the $\Cc^\ell$ norms.
However, in the last term
there appears originally the product
$\p_t\p_tu$ times $\p_t\p_tu$.
To deal with this product we applied
the first estimate of
corollary~\ref{cor:product-Sobolev}
(where in both factors
$u$ is replaced by $\p_tu$).

\noindent
Now the $\Cc^\ell$
estimate~(\ref{eq:C^0-p_tu})
for $\p_tu$ and
the $\Ww^{\ell+1,p}$ estimate
for $\p_tu$ established
in~${\text{(a)}}_{\ell+1}$ above
prove that
$$
     \Norm{\p_t\p_th}
     _{\Ww^{\ell,p}(Z_{3\ell+3})}
     \le\gamma_{\ell+1}
     +\Norm{\p_t\p_tF}_{\Ww^{\ell,p}}
$$
for some constant
$\gamma_{\ell+1}
=\gamma_{\ell+1}(\ell,p,\mu_0,
\norm{\Gamma}_{C^{2\ell+2}},
\norm{F}_{\Ww^{\ell,p}},
\norm{\p_tF}_{\Ww^{\ell,p}})$.
Apply again
theorem~\ref{thm:interior-regularity}
to see that
$\p_t\p_tu\in\Ww^{\ell+1,p}_{loc}(Z_{3\ell+3})$
and
$$
     \Norm{\p_t\p_tu}
     _{\Ww^{\ell+1,p}(Z_{3\ell+4}))}
     \le \mu\left(
     \Norm{\p_t\p_th}
     _{\Ww^{\ell,p}(Z_{3\ell+3})}
     +\Norm{\p_t\p_tu}_{L^p(Z_{3\ell+3})}
     \right)
$$
for some constant
$\mu=\mu(p,Z_{3\ell+3},Z_{3\ell+4})$.
The estimate for $\p_t\p_th$
then proves~(c)
in the case $\ell+1$.
This completes the proof 
of the induction step
and therefore of the claim.
The claim with $\ell=k$
proves proposition~\ref{prop:bootstrap}.
\end{proof}

The following product estimates
have been used in
the parabolic bootstrap
iteration above.
Recall the definition~(\ref{eq:parabolic-Ck})
of the parabolic $\Cc^k$ norm.

\begin{lemma}\label{le:product-Sobolev}
Fix a constant $p>2$ and
a bounded open subset
$\Omega\subset\R^2$ with area $\abs{\Omega}$.
Then for every 
integer $k\ge 1$
there is a constant $c=c(k,\abs{\Omega})$
such that
$$
     \Norm{\p_tu\cdot v}_{\Ww^{k,p}}
     \le c \left(\Norm{\p_tu}_{\Ww^{k,p}}
     \Norm{v}_\infty
     +\Norm{u}_{\Cc^k}
     \Norm{v}_{\Ww^{k,p}}\right)
$$
for all functions $u,v\in
C^\infty(\overline{\Omega})$.
\end{lemma}

\begin{proof}
The proof is by induction on $k$.
By definition of the
$\Ww^{\ell,p}$ norm
\begin{equation}\label{eq:ell+1}
\begin{split}
     \Norm{\p_tu\cdot v}_{\Ww^{\ell+1,p}}
    &\le 
     \Norm{\p_tu\cdot v}_{\Ww^{\ell,p}}
     +\Norm{\p_t\p_tu\cdot v+\p_tu\cdot\p_tv}_{\Ww^{\ell,p}}\\
    &\quad+\Norm{\p_t\p_t\p_tu\cdot v+2\p_t\p_tu\cdot\p_tv
     +\p_tu\cdot\p_t\p_tv}_{\Ww^{\ell,p}}\\
    &\quad
     +\Norm{\p_s\p_tu\cdot v+\p_tu\cdot\p_sv}_{\Ww^{\ell,p}}.
\end{split}
\end{equation}

\vspace{.05cm}
\noindent
{\it Case $k=1$.}
Estimate~(\ref{eq:ell+1}) for $\ell=0$ shows that
\begin{equation*}
\begin{split}
     \Norm{\p_tu\cdot v}_{\Ww^{1,p}}
    &\le\bigl(\Norm{\p_tu}_p
     +\Norm{\p_t\p_tu}_p+\Norm{\p_t\p_t\p_tu}_p
     +\Norm{\p_s\p_tu}_p\bigr)\Norm{v}_\infty\\
    &\quad+\bigl(\Norm{\p_tu}_\infty
     +2\Norm{\p_t\p_tu}_\infty\bigr) \Norm{\p_tv}_p\\
    &\quad+\Norm{\p_tu}_\infty\bigl(\Norm{\p_t\p_tv}_p
     +\Norm{\p_sv}_p\bigr)
\end{split}
\end{equation*}
and this proves the lemma for $k=1$.

\vspace{.05cm}
\noindent
{\it Induction step
$k\Rightarrow k+1$.}
Consider estimate~(\ref{eq:ell+1}) for $\ell=k$,
then inspect the right hand side term by term
using the induction hypothesis to conclude the proof.
To illustrate this we give full details
for the last term in~(\ref{eq:ell+1}),
namely
\begin{equation*}
\begin{split}
     \Norm{\p_tu\cdot\p_sv}_{\Ww^{k,p}}
    &\le c\left(\Norm{\p_tu}_{\Ww^{k,p}}
     \Norm{\p_sv}_\infty
     +\Norm{u}_{\Cc^k}
     \Norm{\p_sv}_{\Ww^{k,p}}\right)\\
    &\le c\left(c^\prime\Abs{\Omega} \Norm{\p_tu}_{\Cc^k}
     \Norm{\p_sv}_{\Ww^{1,p}}
     +\Norm{u}_{\Cc^k}
     \Norm{v}_{\Ww^{k+1,p}}\right)\\
    &\le c\left(c^\prime\Abs{\Omega} \Norm{u}_{\Cc^{k+1}}
     \Norm{v}_{\Ww^{2,p}}
     +\Norm{u}_{\Cc^k}
     \Norm{v}_{\Ww^{k+1,p}}\right).
\end{split}
\end{equation*}
Step one is by the induction
hypothesis. In step two
we pulled out the $L^\infty$ norms
of all derivatives of $\p_tu$
and for the term $\p_sv$
we used the Sobolev embedding
$\Ww^{1,p}\subset W^{1,p}\hookrightarrow C^0$
with constant $c^\prime$.
Here the assumptions $p>2$
and $\Omega$ bounded enter.
Step three is obvious.
Now
$\Ww^{k+1,p}\hookrightarrow\Ww^{2,p}$
since $k\ge1$.
\end{proof}

\begin{corollary}\label{cor:product-Sobolev}
Fix a constant $p>2$ and
a bounded open subset $\Omega\subset\R^2$.
Then for every 
integer $k\ge 1$
there is a constant $d=d(k,\abs{\Omega})$
such that
\begin{align*}
     \Norm{\p_tu\cdot\p_tu}_{\Ww^{k,p}}
    &\le d_k\Norm{u}_{\Cc^k}
     \Norm{\p_tu}_{\Ww^{k,p}}
 \\
     \Norm{\p_tu\cdot\p_t\p_tu}_{\Ww^{k,p}}
    &\le d_k \Norm{u}_{\Cc^k}
     \left(
     \Norm{\p_tu}_{\Ww^{k,p}}
     +\Norm{\p_t\p_tu}_{\Ww^{k,p}}
     \right)
 \\
     \Norm{\p_tu\cdot\p_tu\cdot\p_tu}_{\Ww^{k,p}}
    &\le d_k\Norm{u}_{\Cc^k}^2
     \Norm{\p_tu}_{\Ww^{k,p}}
\end{align*}
for every function $u\in
C^\infty(\overline{\Omega})$.
\end{corollary}

\begin{proof}
All three estimates
follow from
lemma~\ref{le:product-Sobolev}.
To obtain the first 
and the second estimate 
set $v=\p_tu$ and
$v=\p_t\p_tu$, respectively,
and use that
$$
     \Norm{\p_tu}_\infty
     \le \Norm{u}_{\Cc^k}
     ,\qquad
     \Norm{\p_t\p_tu}_\infty
     \le \Norm{u}_{\Cc^k}.
$$
To obtain
the third estimate
set $v=\p_tu\cdot\p_tu$
and use in addition the
first estimate of
corollary~\ref{cor:product-Sobolev}.
\end{proof}

\begin{proof}[Proof of theorem~\ref{thm:bootstrap}]
The $\Ww^{k+1,p}$ norm of $u$
is equivalent to the 
sum of the $\Ww^{k,p}$ norms
of $u$, $\p_tu$, $\p_su$, and $\p_t\p_tu$.
Apply
proposition~\ref{prop:bootstrap}~(i--iii).
\end{proof}


\end{document}